\documentclass[preprint,authoryear,12pt]{article}
\usepackage{amsmath}
\usepackage{amssymb}
\usepackage{graphics}

\usepackage{graphicx}
\usepackage{setspace}
\usepackage{enumerate}
\usepackage{natbib}
\usepackage{amsthm}
\usepackage{amsmath}
\usepackage{amssymb}
\usepackage{subcaption}
\usepackage{arydshln}
\usepackage{epstopdf}
\usepackage{float}
\usepackage[mathscr]{euscript}
\usepackage{mathtools}
\usepackage{ragged2e}

\usepackage{setspace}
\usepackage{enumerate}
\usepackage{natbib}
\usepackage{multirow}

\newtheorem{theorem}{Theorem}
\newtheorem{lemma}{Lemma}

\title{Title}

\author{}
\date{}
\usepackage[bottom=1.2in]{geometry}
\begin{document}
{
  \title{\bf ARMA Models for Zero Inflated Count \\Time Series}
  \author{ Vurukonda Sathish $^a$,  Siuli Mukhopadhyay {\footnote {Corresponding author. Department of Mathematics, Indian Institute of Technology Bombay, Mumbai 400 076, India. Email: siuli@math.iitb.ac.in}} $^{b}$ and Rashmi Tiwari $^b$ \hspace{.2cm}\\
    $^a$ Department of Electrical Engineering\\
    $^b$ Department of Mathematics\\ Indian Institute of Technology Bombay, India}
  \maketitle
}



%
{\doublespacing
\begin{abstract}


Zero inflation is a common nuisance while monitoring disease progression over time.  This article proposes a new observation driven model for zero inflated and over-dispersed count time series.  The counts given the past history of the process and available information on covariates are assumed to be distributed as a mixture of a Poisson distribution and a distribution degenerate at zero, with a time dependent mixing probability,  $\pi_t$. Since, count data usually suffers from overdispersion,  a Gamma distribution is used to model the excess variation, resulting in a zero inflated negative binomial (NB) regression model with mean parameter $\lambda_t$. Linear predictors with auto regressive and moving average (ARMA) type terms, covariates, seasonality and trend are fitted to $\lambda_t$ and $\pi_t$ through canonical link generalized linear models. Estimation is done using maximum likelihood  aided by iterative algorithms, such as Newton Raphson (NR) and Expectation and Maximization (EM).  Theoretical results on the consistency and asymptotic normality of the estimators are given. 
The proposed model is illustrated using in-depth simulation studies and a dengue data set. 
\end{abstract}	

\noindent%

\section{Introduction}

Time series modelling of  counts though a popularly researched topic, is often plagued by the problem of  excess zeros. For example, during geographical mapping  of the prevalence of a certain disease, it is commonly noted that  some regions/areas under study report a high degree of zero disease counts. In such situations, it is difficult to understand if the observed counts, represent  true disease prevalence or a spurious one (i.e., the disease may be present but was not observed/reported). 
However, ignoring the excess zero counts while setting up the time series model may lead to incorrect estimates and loss of significant results \citep{chaney}. 

The two main methods for modeling count time series are observation driven and parameter driven models (\cite{cox}). These two modelling techniques differ from each other in the way they account for the  autocorrelation in the data. While in observation driven models, the  autocorrelation is modeled as a function of the past responses, a latent variable approach is used in parameter driven models.  For nonnormal  time series, parameter driven models were considered by \cite{west}, \cite{fahrwagen}, \cite{durbin1997,durbin2000}, \cite{gamerman1998}, \cite{fahrmeir}, \cite{chiogna}, \cite{godolphin}, \cite{gamerman2013}, among others. 
To compute the posterior distributions of the parameters in such models, MCMC methods (\cite{durbin1997}, \cite{sp} and  \cite{gamerman1998}) were frequently used  for  estimation. \cite{Benjamin2003} noted that often these MCMC algorithms fail to converge resulting in poor inference and predictions. Also,  in more recent times the particle filter algorithm (\cite{dk})  has been sometimes used to replace these MCMC methods.   Compared to parameter driven models, the computational complexity for parameter estimation   is much less in observation-driven models (\cite{10.2307/2291149},
1995; \cite{durbin2000},  \cite{Davis2003}). In these models, the conditional distribution for each observation given  past information on responses and past and possibly present covariates  is described by a generalized linear model (GLM) distribution. Partial likelihood theory combined with GLMs  is used for  estimation. These type of models for
count time series have been used by \cite{ZegerQaqish1988},
\cite{Li1994}, \cite{fokianos2003}, 
\cite{Shephard1995}, \cite{Davis2003},
and \cite{Benjamin2003},  \cite{fokianos2009}, \cite{fokianos2011}, \cite{douc2013} and \cite{davis2016}, among others.

In contrast to a large body of work on count time series models, literature on such time series with zero inflation is still very sparse. \cite{chaney} argued the importance of using a zero inflated model, and  showed that  ignoring these excess zeros may lead to  poor estimation and chances of statistically significant findings being ignored. \cite{yau2004} proposed a mixed autoregressive (AR) model for zero inflated count time data and applied it to an occupational health study. The autocorrelation was modeled using a random effect. More recently, \cite{schmidt2011} used parameter driven models to fit zero inflated count time data, while \cite{yang2013} proposed a zero inflated observation driven AR model. Note, all of these zero inflated models were AR in nature, not considering the possibility of including moving average (MA) components. However, in classical time series literature we have seen that MA terms  play an important role, and approximation of such terms by AR  terms may lead to inaccurate predictions (\cite{Li1994}). Other than AR models, generalized autoregressive conditional heteroskedastic models suited for zero inflated integer-valued time series have also been used by \cite{zhu2012} and \cite{xu2020adaptive}. Very recently, a distribution-free approach for estimation of regression parameters was studied by \cite{ghahramani2020}. 
Other  techniques, like  integer autoregressive (INAR) class of models, for modelling such zero inflated count time series data modeled have also been discussed by \cite{jazi2012} and \cite{qi2019}.

In this article, we propose  a new class of observation driven models for zero inflated and overdispersed time dependent counts. Our approach combines both AR as well as MA components in the zero inflated count time series model. The counts conditioned on the past information on responses and past and possibly present covariates were assumed to follow a mixture of a Poisson distribution and a  distribution  degenerating at zero. Since, the assumption of equality of mean and variance as made by the Poisson distribution rarely holds true in practical situations,  the excess variation is modeled by a  Gamma distribution, resulting in a zero inflated negative binomial (NB)   model. Regression models with ARMA terms, along with trend, seasonality and various other covariates are fitted to the mean parameter of the NB, $\lambda_t$, and the mixing parameter, $\pi_t$, using log and logit links, respectively. The Newton-Raphson (NR) algorithm is used to maximize the conditional log-likelihood function and obtain the maximum likelihood (ML) estimators of the parameters. Due to the complicated expression of the log likelihood, we also develop an  Expectation-Maximization (EM) algorithm to provide alternative parameter estimates.  The proposed model estimators are shown to be consistent and asymptotically normal using the central limit theorem for martingales. Note for count time series ARMA models, no such asymptotic results exist in the available literature.

To motivate the necessity of including both AR and MA terms,  we use our proposed approach  to model the dynamics of dengue in one of the most densely populated cities in India. We show that the zero inflated NB-ARMA model helps us to visualise and understand  the progression of dengue over time more accurately than other models with only AR terms.   Since, dengue incidence has been related to various climatic factors (\cite{jain2019}, \cite{siriyasatien2016}),  we include the climatic factors, relative humidity, amount of rainfall, temperature in our model. A superior predictive performance of our proposed model when compared with various other available techniques for modelling dengue,  establishes the  need for using a zero inflated NB-ARMA regression model. We also fit our proposed model to the popular syphilis data (see \cite{ghahramani2020}), and compare our model's performance with other methods available in the statistical literature.

To summarise, the original contributions of this article include (i) a new ARMA type model for a zero inflated and overdispersed count time series setup, (ii) theoretical proofs of asymptotic properties of the model estimators and (iii) in-depth analysis of an unpublished Indian dengue data.

The rest of the paper is organized as follows. The proposed model is discussed in Section 2. In Section 3 we give the details on model estimation, inference and asymptotic theory. Simulation results for evaluating the estimation techniques are provided in Section 4, followed by  a detailed analysis of the zero inflated dengue data in Section 5 and syphilis data in Section 6.

\section{Proposed Statistical Model}

We model the conditional distribution of the counts at time $t$,  $ Y_{t} $, given $H_t$ and $w$, by a hierarchical mixture distribution,
%
\begin{equation}\label{A1}
Y_{t}|H_{t},w\sim \newcommand{\twopartdef}[4]

\left\{
\begin{array}{ll}
0 & \mbox{with probability } \pi_{t} \\
\text{Poisson}(\lambda_t w) & \mbox{with probability}(1-\pi_{t})
\end{array}
\right.
\end{equation}
where $w$ follows  $\text{Gamma}(k ,k)$. Here  $H_t$ is the  information available on responses till time $ (t-1) $ and on covariates till time $t$, and $w$ represents the excess variation for $t=1,\ldots,N$. 
The above equation can be re-expressed as, 
 

\begin{equation}\label{A2}
Y_{t}|H_{t}\sim \newcommand{\twopartdef}[4]

\left\{
\begin{array}{ll}
0 & \mbox{with probability } \pi_{t} \\
\text{NB}(k,\frac{1}{1+\lambda_{t}/k}) & \mbox{with probability }(1-\pi_{t}).
\end{array}
\right.
\end{equation}
or equivalently as,
\begin{equation}\label{A2.1}
Y_{t}|H_{t}=\newcommand{\twopartdef}[4]

\left\{
\begin{array}{ll}
0 & \mbox{with probability } \pi_{t}+(1-\pi_{t})(\frac{1}{1+\lambda_{t}/k})^{k}~; \\
m & \mbox{with probability } (1-\pi_{t})\frac{\Gamma(m+k)}{m!~\Gamma(k)}\Big(\frac{\lambda_{t}/k}{1+\lambda_{t}/k}\Big)^{m}\Big(\frac{1}{1+\lambda_{t}/k}\Big)^{k}
\end{array}
\right.
\end{equation}
where $m=1,2,\ldots$.	
\noindent The conditional mean and variance of $Y_{t}|H_{t}$ are $E(Y_{t}|H_{t}) =\lambda_{t}(1-\pi_{t}) =\Lambda_{t}$ and $\text{Var}(Y_{t}|H_{t}) =\lambda_{t}(1-\pi_{t})[1+\lambda_{t}\pi_{t}+\lambda_{t}/k]=\Psi_{t}$, respectively.
%
%


As in a standard GLM setup the means are  related to the linear predictor through an  invertible link function $g(\cdot)$. In the zero inflated model, there are two means, the NB mean, $\lambda_t$, and  the Bernoulli mean (mixing parameter), $\pi_t$, which are respectively modeled  as, 
 \begin{eqnarray*} \lambda_{t}&=&\exp(W_t)\\\pi_{t}&=&\frac{\exp(M_t)}{1+\exp(M_t)}.
\end{eqnarray*}

The state processes (linear predictors),  
\begin{eqnarray}\label{state_process}
W_{t}&=&x^{T}_{t}\beta+Z_{t}\nonumber\\M_{t}&=&u^{T}_{t}\delta+V_{t}\end{eqnarray}
are assumed to depend on covariates $x_{t}$ and $u_{t}$ having regression functions possibly involving: 
\begin{enumerate}
	\item Trend type terms: These regression functions have the form $x_{t}=f_{1}(t'/(N-1))$ and $u_{t}=f_{2}(t'/(N-1))$ where $f_{1}(\cdot)$ and $f_{2}(\cdot)$ are real valued functions defined on the interval $[0,1]$ and $t'=t-1$. For example, for linear trend,  $f_{1}^{T}(t'/(N-1))=[1~~t'/(N-1)]$ and $f_{2}^{T}(t'/(N-1))=[1~~t'/(N-1)]$. It is necessary to divide time by sample size, to exclude cases where the NB (Binomial) mean may be near zero or infinity (one) for large sample sizes when trend coefficient is negative or positive, respectively.
	\item Harmonic terms: These terms specify annual/half yearly/monthly/quarterly/\\weekly effects. For example, $x_{t}^{T}=[cos(2\pi t'/12)~~sin(2\pi t'/12)]$ and \\$u_{t}^{T}=[cos(2\pi t'/12)~~sin(2\pi t'/12)]$ where $t'=t-1$.
	\item Stationary processes: These terms may specify  weather series, such as humidity, rain fall, temperature etc. 
\end{enumerate}
Covariates  $x_{t}$ and $u_{t}$  may be same or different for the two linear predictors, Autoregressive and moving average terms to be added to the linear predictor are as follows: For $ t\leq 0 $, we define $ V_t=Z_t=e_{t}=0$ and for $t>0$, 
\begin{eqnarray}\label{A4.1}
Z_{t}&=&\sum_{i=1}^{p_1}\phi_{i}(Z_{t-i}+e_{t-i})+\sum_{j=1}^{q_1}\theta_{j}e_{t-j},
\\\label{A4.2}
V_{t}&=&\sum_{i=1}^{p_2}\alpha_{i}(V_{t-i}+e_{t-i})+\sum_{j=1}^{q_2}\gamma_{j}e_{t-j},
\end{eqnarray}
where zeros of the polynomials $\phi(z)=1-\sum_{i=1}^{p_{1}}\phi_{i}z^{i}, \theta(z)=1+\sum_{i=1}^{q_{1}}\theta_{i}z^{i}, \alpha(z)=1-\sum_{i=1}^{p_{2}}\alpha_{i}z^{i}$ and $\gamma(z)=1+\sum_{i=1}^{q_{2}}\gamma_{i}z^{i}$ are outside of the unit circle. The standardized error $e_t$ is defined as
\begin{equation}\label{eroor}
e_{t}=\frac{ (Y_{t}-\Lambda_{t})} {\sqrt{\Psi_{t}}}.
\end{equation}

The ARMA model gives $H_t=(x_1,\ldots,x_{t},u_1,\ldots,u_t,y_1,\ldots,y_{t-1},\Lambda_1,\ldots,\Lambda_{t-1},\\\Psi_1,\ldots,\Psi_{t-1})$. In general, we assume that the lengths of $x_t$ and $u_t$ are $n_1$ and $n_2$, respectively, while the  unknown parameters $ \beta $, $ \phi $, $ \theta $, $ \delta $, $ \alpha $ and $ \gamma $ are vectors of lengths $ n_{1}, p_{1}, q_{1}, n_{2}, p_{2}$ and $ q_{2}$, respectively.
\subsection{Some properties of the state processes}
The ARMA models for $Z_t$ and $V_t$ can be re-written as the following $MA(\infty)$ models \citep{Davis2003},
\begin{eqnarray}\label{A4.1MA}
Z_{t}&=&\sum_{j=1}^{\infty}\theta_{j}e_{t-j}\\
\label{A4.2MA}
V_{t}&=&\sum_{j=1}^{\infty}\gamma_{j}e_{t-j}.
\end{eqnarray}
Now, for $t\geq 1$, the standardized error, $e_t$, has expectation $0$ and variance 1.  Also, the independence of  $Y_t|H_t$ makes the  errors  uncorrelated, i.e.,  $Cov(e_t,e_s)=0,\,t\neq s$. Using these  error properties, we obtain,
\begin{align*}
E(W_{t})=x^{T}_{t}\beta,~~~~Var(W_t)=\sum_{j=1}^{\infty}\theta^{2}_{j}=c_{1} \end{align*}
and, for $s=t+h, h>0,$
$cov(W_t,W_s) = \sum_{j=1}^{\infty}\theta_{j}\theta_{j+h}=c_{2}. 
$
Similarly,
\begin{equation*}
E(M_{t})=u^{T}_{t}\delta, ~~~~~~Var(M_{t})= \sum_{j=1}^{\infty}\gamma_{j}^{2} =c_{3}, ~~~~~~cov(M_t,M_s) = \sum_{j=1}^{\infty}\gamma_{j}\gamma_{j+h}=c_{4},
\end{equation*}
where $c_{1},c_{2}$, $c_{3}$ and $c_4$'s are constants. Hence, the variances and covariances of the two state processes are  time independent.

For the simplest case of a MA(1) model for both $W_t$ and $M_t$, it can be shown that the bivariate Markov chain $R_t=(W_{t},M_t)$ is bounded in probability, and therefore has at least one stationary distribution. The proof of this result is given in Supplementary materials. However, note that extension of this result to higher order models is very complex and is not shown here. 

\section{Estimation} 
Model fitting is done by ML using iteratively reweighted least squares, as in the context of standard GLMs. As the covariates maybe stochastic, 
from  the conditional density of $Y_t|H_t$  in  (\ref{A2.1}) we get the partial log-likelihood function to be \citep{fokianos2003},
\begin{eqnarray}
\nonumber&&PL(\Theta)=\sum_{t=1}^{N}PL_{t}(\Theta) =\sum_{y_{t=0}} \log (\pi_{t}+(1-\pi_{t})\tilde{p}_t^{k} ) 
+\sum_{y_{t>0}}\Big[\log(1-\pi_{t}) \\&&+ \log\Gamma(k+y_{t}) - \log\Gamma(k) -\log y_{t}! +k\log \tilde{p}_t + y_{t}\log (1-\tilde{p}_t)\Big],
\label{A5}
\end{eqnarray}
where $\Theta=(\beta,\phi, \theta, \delta,\alpha,\gamma, k)$ is the set of all $p(=n_1+p_1+q_1+n_2+p_2+q_2+1)$ parameters. For the purpose of simplifying the notations,  $\tilde{p}_t$ is defined to be $\frac{k}{k+\lambda_{t}}$ and $\nu=(\beta,\phi, \theta, \delta,\alpha,\gamma)$. 
Note, $\nu$ represents the model parameters in the processes $W_t$ and $M_t$, where
\begin{eqnarray*}W_t&=&x^{T}_{t}\beta+\sum_{i=1}^{p_1}\phi_{i}(Z_{t-i}+e_{t-i})+\sum_{j=1}^{q_1}\theta_{j}e_{t-j}\\M_{t}&=&u^{T}_{t}\delta+\sum_{i=1}^{p_2}\alpha_{i}(V_{t-i}+e_{t-i})+\sum_{j=1}^{q_2}\gamma_{j}e_{t-j},\end{eqnarray*} 
while $k$ expresses the overdispersion. The  score functions are as follows:
\begin{eqnarray}
&&\frac{\partial PL}{\partial \nu}=\sum_{y_{t}=0}\frac{1} {\pi_{t}+(1-\pi_{t})\tilde{p}_t^{k} } \Big[\frac{\partial\pi_{t}}{\partial\nu}+ (1-\pi_{t})k \tilde{p}_t^{k-1}\frac{\partial \tilde{p}_t}{\partial \nu} - \tilde{p}_t^{k}\frac{\partial\pi_{t}}{\partial\nu}\Big]\nonumber \\
&+&\sum_{y_{t}>0}\Big[-\frac{1}{1-\pi_{t}}\frac{\partial \pi_{t} }{\partial \nu} + \frac{k}{\tilde{p}_t}\frac{\partial \tilde{p}_t }{\partial \nu} - \frac{y_{t}}{1-\tilde{p}_t}\frac{\partial \tilde{p}_t }{\partial \nu}\Big],
\label{fdnu}\\\nonumber\\
&&\frac{\partial PL}{\partial k}=\sum_{y_{t}=0}\frac{1} {\pi_{t}+(1-\pi_{t})\tilde{p}_t^{k} } \Big[\frac{\partial\pi_{t}}{\partial k}+ (1-\pi_{t})\tilde{p}_t^{k}\Big(\frac{k}{\tilde{p}_t}\frac{\partial \tilde{p}_t}{\partial k}+\log \tilde{p}_t\Big) -\tilde{p}_t^{k}\frac{\partial\pi_{t}}{\partial k}\Big]\nonumber \\
&+&\sum_{y_{t}>0}\Big[-\frac{1}{1-\pi_{t}}\frac{\partial \pi _{t} }{\partial k} 
+\psi_{0}(k+y_{t}) - \psi_{0}(k)+
 \frac{k}{\tilde{p}_t}\frac{\partial \tilde{p}_t }{\partial k} 
 + \log \tilde{p}_t
 - \frac{y_{t}}{1-\tilde{p}_t}\frac{\partial \tilde{p}_t }{\partial k}\Big],
\label{fdk}
\end{eqnarray}
where $\psi_{0}(\cdot)$ is a digamma function. For details see  Supplementary materials. 

The solutions of the score equations, $\frac{\partial PL}{\partial \nu}=0$ and $\frac{\partial PL}{\partial k}=0$ are denoted respectively by $\hat{\nu}$ and $\hat{k}$, however the estimators do not have closed form expressions necessitating the use of the  NR algorithm as follows,

\begin{equation*}
  \widehat{\Theta}^{(j+1)}=\widehat{\Theta}^{(j)}+\tilde{H}_{N}^{-1}(\widehat{\Theta}^{(j)})S_{N}(\widehat{\Theta}^{(j)}),
  \end{equation*}
where $S_{N}({\Theta})=\frac{\partial PL}{\partial \Theta}$ is the score function and $\tilde{H}_{N}({\Theta})=\frac{\partial^{2} PL}{\partial \Theta \partial \Theta^{T}}$ is the observed information matrix. Detailed derivations of the components of $\tilde{H}_{n}({\Theta})$ are in Supplementary materials. 

\subsection{EM algorithm}\label{EM}
Due to the complicated  score equations, \cite{Lambert1992}  suggested an alternative use of the EM algorithm (\cite{dempster}). Suppose, it is possible to observe a Bernoulli variable $S_t$, which takes value $1 $ when $ Y_{t} $ is from the perfect zero state and is $0$ when $ Y_{t} $ is from the NB state.  Then,  
%
the partial log-likelihood with complete data $ \boldsymbol{y}:=(y_{1}, y_{2}, \dots, y_{N}) $ and $ \boldsymbol{s}:=(s_{1}, s_{2}, \dots, s_{N}) $ can be written as,
\begin{eqnarray}
&&PL^{c}(\Theta) =\sum_{t=1}^{N} \mbox{log}(f(s_{t}|H_{t}))+ \sum_{t=1}^{N} \mbox{log}(f(Y_{t}|s_{t}, H_{t}))\nonumber\\&=&\sum_{t=1}^{N} s_{t}\log (\pi_{t})+(1-s_{t})\log(1-\pi_{t}) \nonumber\\
&+ &\sum_{t=1}^{N}(1-s_{t})\Big[\log\Gamma(k+y_{t}) - \log\Gamma(k) -\log y_{t}! +k\log \tilde{p}_t + y_{t}\log (1-\tilde{p}_t)\Big].
\end{eqnarray}
In the E-step 
the expectation of $ PL_{t}^{c}(\Theta) $ is computed with respect to the observed count series  $\boldsymbol{y}$ and current parameter estimates. The $(i+1)$th iteration yields,
\begin{align}
Q(\Theta|\Theta^{(i)})=&E\{PL^{c}(\Theta)|\boldsymbol{y},\Theta^{(i)}\}=\sum_{t=1}^{N} \hat{s}^{(i)}_{t}\log (\pi_{t})+(1-\hat{s}^{(i)}_{t})\log(1-\pi_{t}) \nonumber\\
+ \sum_{t=1}^{N}&(1-\hat{s}^{(i)}_{t})\Big[\log\Gamma(k+y_{t}) - \log\Gamma(k) -\log y_{t}! +k\log \tilde{p}_t + y_{t}\log (1-\tilde{p}_t)\Big],
\end{align}
\begin{eqnarray*}
\text{where }\hat{s}^{(i)}_{t} &=&\text{Pr}(s_{t}=1|Y_{t}=y_{t},H_t,\Theta^{(i)})\\&=&\frac{\text{Pr}(s_{t}=1|\Theta^{(i)})\text{Pr}(Y_{t}=y_{t}|H_t,s_{t}=1,\Theta^{(i)})}{\text{Pr}(Y_{t}=y_{t}|H_t,\Theta^{(i)})}\\&=&\frac{\pi_{t}^{(i)}I_{\{y_{t}=0\}}}{\pi_{t}^{(i)}I_{\{y_{t}=0\}}+(1-\pi_{t}^{(i)})\frac{\Gamma(y_{t}+k)}{y_{t}!\Gamma(k)}\Big(\frac{\lambda_{t}^{(i)}/k}{1+\lambda_{t}^{(i)}/k}\Big)^{y_{t}}\Big(\frac{1}{1+\lambda_{t}^{(i)}/k}\Big)^{k}}.			
\end{eqnarray*} 

In the {M-step}, $ \Theta^{(i+1)} $ is found by maximizing $ Q(\Theta;\Theta^{(i)}) $.
The NR algorithm estimates $ \Theta^{(k+1)} $ as described at the end of the previous section.
The detailed derivations of the components $T_{n}({\Theta})=\frac{\partial Q(\Theta;\Theta^{(i)})}{\partial \Theta }$ and $J_{n}({\Theta})=\frac{\partial^{2} Q(\Theta;\Theta^{(i)})}{\partial \Theta \partial \Theta'}$ are given in Supplementary materials. The sequence $ \{\Theta_{j}: j=1,2,\dots\} $ generated by the EM algorithm is a convergent sequence and it converges to local maximizer of  (\ref{A5}).

\subsection{Asymptotic Theory}
The consistency and asymptotic normality of the maximum partial likelihood estimator $\hat{\Theta}=[\hat{\nu}^{T}~~\hat{k}]^{T}$ is addressed in this section. We rewrite the score functions \eqref{fdnu} as $S_{N}(\nu)=\sum_{t=1}^{N}S^{t}(\nu)$ and $S_{N}(k)=\sum_{t=1}^{N}S^{t}(k)$ where
\begin{align}\label{score1}
S^{t}(\nu)=&\frac{y_{0t}} {\pi_{t}+(1-\pi_{t})\tilde{p}_t^{k}} \left[\frac{\partial\pi_{t}}{\partial\nu}+ (1-\pi_{t})k \tilde{p}_t^{k-1}\frac{\partial \tilde{p}_t}{\partial \nu} - \tilde{p}_t^{k}\frac{\partial\pi_{t}}{\partial\nu}\right]\nonumber \\+&(1-y_{0t})\left[-\frac{1}{1-\pi_{t}}\frac{\partial \pi_{t} }{\partial \nu} + \frac{k}{\tilde{p}_t}\frac{\partial \tilde{p}_t }{\partial \nu} - \frac{y_{t}}{1-\tilde{p}_t}\frac{\partial \tilde{p}_t }{\partial \nu}\right],
\end{align}
and
\begin{align}\label{score2}
S^{t}(k)=&\frac{y_{0t}} {\pi_{t}+(1-\pi_{t})\tilde{p}_t^{k}} \left[\frac{\partial\pi_{t}}{\partial k}+ (1-\pi_{t})\tilde{p}_t^{k}\Big(\frac{k}{\tilde{p}_t}\frac{\partial \tilde{p}_t}{\partial k}+\log \tilde{p}_t\Big) -\tilde{p}_t^{k}\frac{\partial\pi_{t}}{\partial k}\right]\nonumber \\+(1-y_{0t}&)\left[-\frac{1}{1-\pi_{t}}\frac{\partial \pi _{t} }{\partial k} 
+\psi_{0}(k+y_{t}) - \psi_{0}(k)+
\frac{k}{\tilde{p}_t}\frac{\partial \tilde{p}_t }{\partial k} 
+ \log \tilde{p}_t
- \frac{y_{t}}{1-\tilde{p}_t}\frac{\partial \tilde{p}_t }{\partial k}\right],
\end{align}
here $y_{0t}:=I(y_{t}=0)$ is an indicator function. Let us consider $E\left[S^{t}(\nu)|H_{t}\right]$ and $E\left[S^{t}(k)|H_{t}\right] $ which help in proving the consistency and asymptotic normality of the maximum partial likelihood estimator $\hat{\Theta}$.
\begin{align}\label{expscr17}
E\left[S^{t}(\nu)|H_{t}\right] =& E\Big\{\frac{y_{0t}} {\pi_{t}+(1-\pi_{t})\tilde{p}_t^{k}} \left[\frac{\partial\pi_{t}}{\partial\nu}+ (1-\pi_{t})k \tilde{p}_t^{k-1}\frac{\partial \tilde{p}_t}{\partial \nu} - \tilde{p}_t^{k}\frac{\partial\pi_{t}}{\partial\nu}\right]\nonumber \\&\hspace{0.2cm}+(1-y_{0t})\left[-\frac{1}{1-\pi_{t}}\frac{\partial \pi_{t} }{\partial \nu} + \frac{k}{\tilde{p}_t}\frac{\partial \tilde{p}_t }{\partial \nu} - \frac{y_{t}}{1-\tilde{p}_t}\frac{\partial \tilde{p}_t }{\partial \nu}\right]\vert H_{t}\Big\}\nonumber\\
= &\frac{E[y_{0t}|H_{t}]} {\pi_{t}+(1-\pi_{t})\tilde{p}_t^{k}} \left[\frac{\partial\pi_{t}}{\partial\nu}+ (1-\pi_{t})k \tilde{p}_t^{k-1}\frac{\partial \tilde{p}_t}{\partial \nu} - \tilde{p}_t^{k}\frac{\partial\pi_{t}}{\partial\nu}\right]\nonumber \\&\hspace{0.2cm}+E[(1-y_{0t})|H_{t}]\left[-\frac{1}{1-\pi_{t}}\frac{\partial \pi_{t} }{\partial \nu} + \frac{k}{\tilde{p}_t}\frac{\partial \tilde{p}_t }{\partial \nu}\right] - \frac{E[y_{t}|H_{t}]}{1-\tilde{p}_t}\frac{\partial \tilde{p}_t }{\partial \nu}.
\end{align}
Substituting $E[y_{0t}|H_{t}] = \pi_{t}+(1-\pi_{t})\tilde{p}_t^{k}$, $E[(1-y_{0t})|H_{t}] = (1-\pi_{t})(1-\tilde{p}_t^{k})$ and $E[y_{t}|H_{t}] = \lambda_{t}(1-\pi_{t})$ in \eqref{expscr17} above, we obtain
\begin{align}\label{expscr}
E\left[S^{t}(\nu)|H_{t}\right] =&~\frac{\partial\pi_{t}}{\partial\nu}+ (1-\pi_{t})k \tilde{p}_t^{k-1}\frac{\partial \tilde{p}_t}{\partial \nu} - \tilde{p}_t^{k}\frac{\partial\pi_{t}}{\partial\nu}\nonumber \\&\hspace{0.2cm}+(1-\pi_{t})(1-\tilde{p}_t^{k})\left[-\frac{1}{1-\pi_{t}}\frac{\partial \pi_{t} }{\partial \nu} + \frac{k}{\tilde{p}_t}\frac{\partial \tilde{p}_t }{\partial \nu}\right] - \frac{\lambda_{t}(1-\pi_{t})}{1-\tilde{p}_t}\frac{\partial \tilde{p}_t }{\partial \nu}\nonumber\\=&~ 0.
\end{align}
Also,
\begin{align}\label{expscr2}
E\left[S^{t}(k)|H_{t}\right] = E\Big\{&\frac{y_{0t}} {\pi_{t}+(1-\pi_{t})\tilde{p}_t^{k}}\left[\frac{\partial\pi_{t}}{\partial k}+ (1-\pi_{t})\tilde{p}_t^{k}\Big(\frac{k}{\tilde{p}_t}\frac{\partial \tilde{p}_t}{\partial k}+\log \tilde{p}_t\Big) -\tilde{p}_t^{k}\frac{\partial\pi_{t}}{\partial k}\right]\nonumber \\
+(1-y_{0t})\Big[-\frac{1}{1-\pi_{t}}&\frac{\partial \pi _{t} }{\partial k} 
+\psi_{0}(k+y_{t}) - \psi_{0}(k)+
\frac{k}{\tilde{p}_t}\frac{\partial \tilde{p}_t }{\partial k} 
+ \log \tilde{p}_t
- \frac{y_{t}}{1-\tilde{p}_t}\frac{\partial \tilde{p}_t }{\partial k}\Big]|H_{t}\Big\},
\nonumber\\
= &\frac{E[y_{0t}|H_{t}]} {\pi_{t}+(1-\pi_{t})\tilde{p}_t^{k}} \left[\frac{\partial\pi_{t}}{\partial k}+ (1-\pi_{t})\tilde{p}_t^{k}\Big(\frac{k}{\tilde{p}_t}\frac{\partial \tilde{p}_t}{\partial k}+\log \tilde{p}_t\Big) -\tilde{p}_t^{k}\frac{\partial\pi_{t}}{\partial k}\right]\nonumber \\&\quad+E[(1-y_{0t})|H_{t}]\left[-\frac{1}{1-\pi_{t}}\frac{\partial \pi _{t} }{\partial k} 
+\frac{k}{\tilde{p}_t}\frac{\partial \tilde{p}_t }{\partial k} 
+ \log \tilde{p}_t\right] \nonumber\\&\quad+ E[(1-y_{0t})\left[\psi_{0}(k+y_{t}) - \psi_{0}(k)\right]|H_{t}] - \frac{E[y_{t}|H_{t}]}{1-\tilde{p}_t}\frac{\partial \tilde{p}_t }{\partial k}.
\end{align}
Substituting $E[y_{0t}|H_{t}]$, $E[(1-y_{0t})|H_{t}] $ and $E[y_{t}|H_{t}] $ we get,
\begin{align}\label{expscr3}
E\left[S^{t}(k)|H_{t}\right]
=& (1-\pi_{t})\log \tilde{p}_t + E[(1-y_{0t})\left[\psi_{0}(k+y_{t}) -  \psi_{0}(k)\right]|H_{t}],\nonumber\\
=& (1-\pi_{t})\log \tilde{p}_t + E\left[(1-y_{0t})\sum_{l=1}^{y_{t}-1}\Big(\frac{1}{k+l-1}\Big)|H_{t}\right],\nonumber\\
=& (1-\pi_{t})\log \tilde{p}_t + E\left[(1-y_{0t})log\Big(1+\frac{2y_{t}}{2k-1}\Big)|H_{t}\right]  
\end{align}
The second part $E\left[(1-y_{0t})log\Big(1+\frac{2y_{t}}{2k-1}\Big)|H_{t}\right]$ in \eqref{expscr3} is complicated. Thus, the consistency and asymptotic normality of the maximum partial likelihood estimator $\hat{\nu}$ with known overdispersion parameter $k$ is addressed in this section.

Let $S_{1}^{t}:=\frac{1}{\pi_{t}+(1-\pi_{t})\tilde{p}_t^{k}} \Big[\frac{\partial\pi_{t}}{\partial\nu}+ (1-\pi_{t})k \tilde{p}_t^{k-1}\frac{\partial \tilde{p}_t}{\partial \nu} - \tilde{p}_t^{k}\frac{\partial\pi_{t}}{\partial\nu}\Big]$ and $S_{2}^{t}:=-\frac{1}{1-\pi_{t}}\frac{\partial \pi_{t} }{\partial \nu} + \frac{k}{\tilde{p}_t}\frac{\partial \tilde{p}_t }{\partial \nu}$. Then, the conditional information matrix is
\begin{align}
G_{N}(\nu) =& \sum_{t=1}^{N}Var[S^{t}(\nu)|H_{t}]\nonumber\\ 
=&\sum_{t=1}^{N}Var\left[y_{0t}S_{1}^{t}
+(1-y_{0t})S_{2}^{t}- \frac{y_{t}}{1-\tilde{p}_t}\frac{\partial \tilde{p}_t }{\partial \nu}|H_{t}\right]\nonumber\\
=&\sum_{t=1}^{N}\Big\{(S_{1}^{t}-S_{2}^{t})(S_{1}^{t}-S_{2}^{t})^{T}Var(y_{0t}|H_{t})+ \frac{Var(y_{t}|H_{t})}{1-\tilde{p}_t}\frac{\partial \tilde{p}_t }{\partial \nu}\big(\frac{\partial \tilde{p}_t }{\partial \nu}\big)^{T}\Big\},
\end{align}
where $Var(y_{0t}|H_{t})=(\pi_{t}+(1-\pi_{t})\tilde{p}_t^{k})(1-\pi_{t})(1-\tilde{p}_t^{k})$ and $Var(y_{t}|H_{t})=\lambda_{t}(1-\pi_{t})(1+\lambda_{t}\pi_{t}+\lambda_{t}/k)$. The unconditional information matrix is $F_{N}(\nu):=E[G_{N}(\nu)]$.

The consistency and the asymptotic normality of $\hat{\nu}$ for known $k$ is proved under the following assumptions, which are slight modifications of those stated in Fokianos and Kedem (\citeyear{fokianos1998,fokianos2003}).
\begin{enumerate}
	\item[A1.] The true parameter $\nu$ belongs to an open set $\mathcal{N}\subseteq\mathbb{R}^{n_1+p_{1}+q_{1}+n_2+p_{2}+ q_{2}}$.
	\item[A2.] The linear predictors $W_{t}$ and $M_{t}$ are geometrically ergodic processes. 
	\item[A3.] Using matrix notations, we write the the linear predictors $(W_t,M_t)$ as 
\begin{eqnarray*}
W_t=\tilde{W'}_t\Theta_1 \text{ and }
M_t=\tilde{M'}_t\Theta_2, 
\end{eqnarray*}
where $\Theta_1=(\beta,\phi,\theta)$ and $\Theta_2=(\delta,\alpha,\gamma)$. 
The covariate vectors $\tilde{W}_t$ and $\tilde{M}_t$ almost surely lies in a nonrandom compact subset $\Gamma_1\subseteq \mathbb{R}^{n_1+p_{1}+q_{1}}$ and $\Gamma_2\subseteq \mathbb{R}^{n_2+p_{2}+q_{2}}$, respectively.
Also, $\tilde{W'}_t\Theta_1$ and $\tilde{M'}_t\Theta_2$ lie almost surely in the domain of their respective inverse link functions (i.e., log and logit) for all $\tilde{W}_t \in \Gamma_1$, $\tilde{M}_t\in \Gamma_2$, $\Theta_1$ and $\Theta_2$.
	\item[A4.] The matrix $G_{N}(\nu)$ is a positive definite matrix with probability 1. Also, $G(\nu)$ (defined below equation \ref{conditional_inf}) is a positive definite matrix at the true values of $\nu$ and therefore its inverse exists.
\end{enumerate}
From assumption A2 and the law of large numbers for geometrically ergodic processes (see \cite{jensen2007law} and \cite{fokianos2009poisson}), 
\begin{align}\label{conditional_inf}
\frac{G_{N}(\nu)}{N}~\xrightarrow[]{\text{~~P~~}}~G(\nu),
\end{align}
as $N\rightarrow\infty$. Thus, the conditional information matrix $G_{N}(\nu)$ has a non-random limit $G(\nu)$.  Here $G(\nu):=\lim_{N\to\infty}\frac{1}{N}\sum_{t=1}^{N}\Big\{(S_{1}^{t}-S_{2}^{t})(S_{1}^{t}-S_{2}^{t})^{T}Var(y_{0t}|H_{t})\\+ \frac{Var(y_{t}|H_{t})}{1-\tilde{p}_t}\frac{\partial \tilde{p}_t }{\partial \nu}\big(\frac{\partial \tilde{p}_t }{\partial \nu}\big)^{T}\Big\}$ $=E\left[(S_{1}^{t}-S_{2}^{t})(S_{1}^{t}-S_{2}^{t})^{T}Var(y_{0t}|H_{t})+ \frac{Var(y_{t}|H_{t})}{1-\tilde{p}_t}\frac{\partial \tilde{p}_t }{\partial \nu}\big(\frac{\partial \tilde{p}_t }{\partial \nu}\big)^{T}\right]$, for geometrically ergodic processes (see \cite{jensen2007law} and \cite{fokianos2009poisson}).

\begin{theorem}\label{score_MDS}
	Consider the state processes, $W_t$ and $M_t$ in \eqref{state_process} and suppose assumptions A1-A4 hold true. Then, 
	$S_{N}(\nu)$ converges in distribution to a normal distribution, $\mathcal{N}(0, G(\nu))$, as $N\rightarrow\infty$. 
\end{theorem}
\begin{proof}
	The central limit theorem (CLT) for martingales (Corollary 3.1 of \cite{hall2014martingale}) will be used to prove the asymptotic normality of   $S_{N}(\nu)$. To use the corollary, we first need to verify the following:
	\begin{enumerate}[i.]
		\item The score function is a martingale difference sequence i.e., $E[S^{t}(\nu)|H_{t}]=0$ and $E[\Vert S^{t}(\nu)\Vert]<\infty$ where $||.||$ is a vector norm (see page 230 of \cite{davidson1994}),
		\item Lindeberg's condition i.e., for all $\epsilon>0$, \\$\frac{1}{N}\sum_{t=1}^{N}E\left[\Vert S^{t}(\nu)\Vert^{2} I(\Vert S^{t}(\nu)\Vert>\epsilon)|H_{t}\right]\xrightarrow[]{~P~}0$, as $N\rightarrow\infty$, and
		\item The normalized conditional information matrix $\frac{G_{N}(\nu)}{N}$ converges in probability to a non-random limit $G(\nu)$.
	\end{enumerate}
From \eqref{expscr}, we know that $E\left[S^{t}(\nu)|H_{t}\right]=0$. Next consider, 
\begin{align}
E\left[\Vert S^{t}(\nu)\Vert\right] \leq &~E\left\lVert\frac{y_{0t}} {\pi_{t}+(1-\pi_{t})\tilde{p}_t^{k}} \left[\frac{\partial\pi_{t}}{\partial\nu}+ (1-\pi_{t})k \tilde{p}_t^{k-1}\frac{\partial \tilde{p}_t}{\partial \nu} - \tilde{p}_t^{k}\frac{\partial\pi_{t}}{\partial\nu}\right]\right\rVert\nonumber \\&\hspace{0.2cm}+E\left\lVert(1-y_{0t})\left[-\frac{1}{1-\pi_{t}}\frac{\partial \pi_{t} }{\partial \nu} + \frac{k}{\tilde{p}_t}\frac{\partial \tilde{p}_t }{\partial \nu} - \frac{y_{t}}{1-\tilde{p}_t}\frac{\partial \tilde{p}_t }{\partial \nu}\right]\right\rVert\nonumber\\
\leq &~E\left[\frac{y_{0t}} {\pi_{t}+(1-\pi_{t})\tilde{p}_t^{k}}\right] E\left\lVert\frac{\partial\pi_{t}}{\partial\nu}+ (1-\pi_{t})k \tilde{p}_t^{k-1}\frac{\partial \tilde{p}_t}{\partial \nu} - \tilde{p}_t^{k}\frac{\partial\pi_{t}}{\partial\nu}\right\rVert\nonumber \\&\hspace{0.2cm}+E[1-y_{0t}]E\left\lVert-\frac{1}{1-\pi_{t}}\frac{\partial \pi_{t} }{\partial \nu} + \frac{k}{\tilde{p}_t}\frac{\partial \tilde{p}_t}{\partial \nu}\right\rVert +  E\left[\frac{y_{t}}{1-\tilde{p}_t}\right]E\left\lVert\frac{\partial\tilde{p}_t }{\partial \nu}\right\rVert,
\end{align}
applying the law of total expectation to the above equation, we get \begin{align}
E\left[\Vert S^{t}(\nu)\Vert\right] \leq &~E\left\lVert\frac{\partial\pi_{t}}{\partial\nu}+ (1-\pi_{t})k \tilde{p}_t^{k-1}\frac{\partial \tilde{p}_t}{\partial \nu} - \tilde{p}_t^{k}\frac{\partial\pi_{t}}{\partial\nu}\right\rVert \nonumber\\&\hspace{-1cm}+ E[(1-\pi_{t})(1-\tilde{p}_{t})]E\left\lVert-\frac{1}{1-\pi_{t}}\frac{\partial \pi_{t} }{\partial \nu} + \frac{k}{\tilde{p}_t}\frac{\partial \tilde{p}_t }{\partial \nu}\right\rVert+E\left[\frac{\lambda_{t}(1-\pi_{t})}{1-\tilde{p}_t}\right]E\left\lVert\frac{\partial \tilde{p}_t }{\partial \nu}\right\rVert
\end{align}
Using assumptions A1 and A2 along with the fact $e_t=0$ for $t\leq 0$ in \eqref{state_process}, $E[W_t]$ and $E[M_{t}]$
are finite,
implying  $E[\lambda_{t}]<\infty$, $E[\pi_{t}]<\infty$. 
Using Holders inequality (\cite{hardy1934inequalities}), we can then conclude that $E\left[\Vert S^{t}(\nu)\Vert\right] <\infty$. Therefore the score function $S^{t}(\nu)$ in \eqref{score1} is a martingale difference sequence. 

For proving Lindeberg's condition, we consider
\begin{align}\label{Lindeberg}
\frac{1}{N}\sum_{t=1}^{N}E\big[\Vert S^{t}(\nu)\Vert^{2} I(\Vert S^{t}(\nu)\Vert>\sqrt{N}\delta)|H_{t}\big]\leq\frac{1}{N^{2}\delta^{2}}\sum_{t=1}^{N}E\big[\Vert S^{t}(\nu)\Vert^{4}|H_{t}\big]\xrightarrow[]{~P~}0,
\end{align}
as $N\rightarrow\infty$, since $E\big[\Vert S^{t}(\nu)\Vert^{4}\big]<\infty$ which follows from assumptions A2, A3 and Holders inequality. The convergence in probability in \eqref{Lindeberg} follows from the law of large numbers for geometrically ergodic processes (\cite{jensen2007law} and \cite{fokianos2009poisson}).

Next, the condition (iii) follows from \eqref{conditional_inf}. Therefore using the CLT for martingales, we can therefore say that the partial score function $S_{N}(\nu)$ converges in distribution to a normal random variable with zero mean and variance $G(\nu)$.
\end{proof}

Expressing the information matrix $\tilde{H}_{N}(\nu)$ as $\tilde{H}_{N}(\nu)=G_{N}(\nu)-R_{N}(\nu)$, where $R_{N}(\nu)$ is a remainder term. Next we will show that the normalized information matrix $\frac{\tilde{H}_{N}(\nu)}{N}$ converges in probability to $G(\nu)$. 
\begin{lemma}\label{continuity}
	Under the assumptions A1-A4, $\frac{\tilde{H}_{N}(\nu)}{N}\xrightarrow{P}G(\nu)$ as $N\rightarrow\infty$.
\end{lemma}
\begin{proof}
	From \eqref{conditional_inf}, we already have $\frac{G_{N}(\nu)}{N}\xrightarrow{P}G(\nu)$. From assumption A2 and the law of large numbers for geometrically ergodic process (\cite{jensen2007law} and \cite{fokianos2009poisson}), $\frac{\tilde{H}_{N}(\nu)}{N}\xrightarrow{P}E\left[\frac{\partial^2 PL_{t}(\Theta)}{\partial \nu\partial \nu^{T}}\right]$. The normalized remainder term $\frac{R_{N}(\nu)}{N}=\frac{G_{N}(\nu)}{N}+\frac{\tilde{H}_{N}(\nu)}{N}$ converges in probability to $G(\nu)+E\left[\frac{\partial^2 PL_{t}(\Theta)}{\partial \nu\partial \nu^{T}}\right]$. Thus, we need to prove that $G(\nu)+E\left[\frac{\partial^2 PL_{t}(\Theta)}{\partial \nu\partial \nu^{T}}\right]$ is zero. 
	In Theorem \ref{score_MDS} we showed that $E[S^{t}(\nu)|H_{t}]=0$, which implies $\sum_{t=1}^{N}S^{t}(\nu)P(Y_{t}|H_{t})=0$. Differentiating both sides with respect to $\nu$ we get,
	$E\left[\frac{\partial^2 PL_{t}(\Theta)}{\partial \nu\partial \nu^{T}}|H_{t}\right]= -E\left[S^{t}(\nu)S^{t}(\nu)^{T}|H_{t}\right]$.
	Thus, we can write
	\begin{align}
	G&(\nu)+E\left[\frac{\partial^2 PL_{t}(\Theta)}{\partial \nu\partial \nu^{T}}\right] \nonumber\\&=E(Var[S^{t}(\nu)|H_{t}]) + E\left[E\left[\frac{\partial^2 PL_{t}(\Theta)}{\partial \nu\partial \nu^{T}}|H_{t}\right]\right]\nonumber\\ &=E(Var[S^{t}(\nu)|H_{t}])  -E\left[E[S^{t}(\nu)(S^{t}(\nu))^{T}|H_{t}]\right]\nonumber\\
	&=E\left[E[S^{t}(\nu)(S^{t}(\nu))^{T}|H_{t}]-E[S^{t}(\nu)|H_{t}]E[S^{t}(\nu)|H_{t}]^{T}\right]-E\left[S^{t}(\nu)(S^{t}(\nu))^{T}\right]\nonumber\\
	&=E\left[E[S^{t}(\nu)(S^{t}(\nu))^{T}|H_{t}]\right]-E\left[E[S^{t}(\nu)|H_{t}]E[S^{t}(\nu)|H_{t}]^{T}\right]-E\left[S^{t}(\nu)(S^{t}(\nu))^{T}\right]\nonumber\\
	&=E\left[S^{t}(\nu)(S^{t}(\nu))^{T}\right]-E\left[S^{t}(\nu)(S^{t}(\nu))^{T}\right]\nonumber\\
	&=0.
	\end{align}
	Hence, the normalized remainder term $\frac{R_{N}(\nu)}{N}$ converges in probability to zero, and  the normalized information matrix $\frac{\tilde{H}_{N}(\nu)}{N}$ converges in probability to $G(\nu)$, as $N\rightarrow\infty$.
\end{proof}
\begin{lemma}\label{min_eig}
	Under the assumptions A1-A4, $\lambda_{min}(F_{N}(\nu))\rightarrow\infty$, as $N\rightarrow\infty$, where $\lambda_{min}(F_{N}(\nu))$ is the minimum eigenvalue of the unconditional information matrix $F_{N}(\nu)$.
\end{lemma}
	Proof is similar to that of Lemma 3.2 of \cite{fokianos1998}.

\begin{theorem}\label{asymptotic}
Under the assumptions A1-A4, the probability that a locally unique maximum partial likelihood estimator exists converges to one. Also, the maximum partial likelihood estimator $\hat{\nu}$ is consistent and \begin{equation*}
	\sqrt{N}(\hat{\nu}-\nu)~\xrightarrow[]{\text{~~D~~}}~\mathcal{N}_{n_1+p_{1}+q_{1}+n_2+p_{2}+ q_{2}}(0, G^{-1}(\nu))
\end{equation*}
as $N\rightarrow\infty$.
\end{theorem}
The outline of the proof is similar to the one in Fokianos and Kedem \citeyear{fokianos1998,fokianos2003}. The details are given in Appendix A.

\subsection{Model Inference and Goodness-of-Fit Statistics}

For inferring about the model parameters,  we make the assumption that the central limit theorem holds for $\widehat{\Theta}$ (for CLT for $\hat{\nu}$ see Theorem 3) 
so that, $\sqrt{n}( \widehat{\Theta}-\Theta) \xrightarrow[\text{}]{\text{d}} N(0,\hat{\Omega})$, as $n\rightarrow \infty$, where $\hat{\Omega}$ is the approximate covariance matrix, i.e., $\hat{\Omega}=-\left(\frac{\partial^2 PL(\hat{\Theta})}{\partial \Theta \partial\Theta'}\right)^{-1}$ or $\hat{\Omega}=-\left(\frac{\partial^2 PL^{c}(\hat{\Theta})}{\partial \Theta \partial\Theta'}\right)^{-1}$. 

To test the hypotheses, $H_0:C\Theta=\zeta,\, \text{ versus } H_1:C\Theta\neq\zeta,$
where the matrix $C$ is of order $m\times p$ and of rank $m(\leq p)$,  an  approximate Wald type statistic,  \begin{equation*}
W=(C\widehat{\Theta}-\zeta)^T[C\hat{\Omega}C^T]^{-1}(C\widehat{\Theta}-\zeta)\sim \chi^{2}_{m},\text { for large } N \text{ under }H_0,
 \end{equation*}
is used.
Consider two zero inflated NB-ARMA models, such that the  number of fitted parameters are $(m_{p_1},m_{p_2})$ and their respective deviances are, $(D_{m_{p_1}},D_{m_{p_2}})$. The likelihood ratio test statistic $L=D_{m_{p_1}}-D_{m_{p_2}}$ for testing between the two models may be assumed to follow a $\chi_{m_{p_1}-m_{p_2}}^{2}$ under $H_0$. 
Other than the deviance, the mean squared error (MSE), the Pearson chi-square, the two information criteria (AIC, BIC), and mean absolute deviation (MAD) may also be used to evaluate and select competing models. For definitions of these various selection criteria, refer to Chapter 1 of \cite{Kedem2002}.
\section{Simulation Results}\label{sim}
In this section  the estimators  and their asymptotic properties  are evaluated through  simulations for various zero inflated NB ARMA models.

The models considered are:
\begin{itemize}
	\item Model 1: Only MA terms are considered in the log and logit mean models:
	\begin{align}
	W_{t} = x_{t}^{T}\beta+\theta_{1}e_{t-1}+\theta_{2}e_{t-2}\qquad\text{and}\qquad M_{t}=u_{t}^{T}\delta+\gamma_{1}e_{t-1} 
	\end{align} 
	\item Model 2: Both AR and MA components are considered in the log and logit mean models:
	\begin{align}
	W_{t} =&     x_{t}^{T}\beta+\phi_{1}(Z_{t-1}+e_{t-1})+\theta_{1}e_{t-1}+\theta_{2}e_{t-2}~~~\text{and}\nonumber\\ M_{t}=& u_{t}^{T}\delta+\alpha_{1}(Z_{t-1}+e_{t-1})+\gamma_{1}e_{t-1} 
	\end{align}
	where $x_{t}^{T}=[\begin{matrix}1&\cos(2\pi t/6)&\sin(2\pi t/6)\end{matrix}]$ and $u_{t}^{T}=[\begin{matrix}1&\cos(2\pi t/6)&\sin(2\pi t/6)\end{matrix}]$ in both models 1 and 2.   
	\item Model 3: Only MA terms are considered in the log mean model:
	\begin{align}
	W_{t} = x_{t}^{T}\beta+\theta_{1}e_{t-1}+\theta_{3}e_{t-3}\qquad\text{and}\qquad M_{t}=u_{t}^{T}\delta 
	\end{align} 
	where $x_{t}^{T}=[\begin{matrix}1~t'/(N-1)~\cos(2\pi t'/52)~\sin(2\pi t'/52)\end{matrix}]$ and \\$u_{t}^{T}=[\begin{matrix}1~\cos(2\pi t'/52)~\sin(2\pi t'/52)\end{matrix}]$.  
\end{itemize}

Simulations with $1000$ replications and sample sizes $N=30,\, 100,\,500$ were run, and the results corresponding to the model 3 are displayed in Table \ref{SimCompTableMAreal}. Tables for models 1 and 2 are put in the Supplementary materials. Goodness of fit plots for all models and $n=100$ are presented in Figure \ref{GF_sim}. From this figure, we note that fitted counts from both EM and MLE are very close to the observed counts. 

For EM estimates (all models), we see that the average bias of the model parameters  range between $(2\%-24\%)$ when $N=30$. As we increase sample size to $N=500$, the average bias  reduces to $(0.07\%-3\%)$. The standard errors range from $(0.2\%-1\%)$ for all values of $N$. The results for MLE are similar. We also note from the simulation results that the bias of the overdispersion parameter $k$ is always positive and are also quite high as compared to the model parameters.

The overestimation of $k$ has been observed before in the statistical literature, see \cite{piegorsch} and has been attributed to  small values of mean parameters and sample sizes.  For improving the bias, it is suggested to estimate the reciprocal of $k$ instead of $k$ itself, this is done to avoid the estimation of a discontinuous variable and also the confidence intervals of $1/k$ are more continuous and symmetric. Also, method of moments estimation and maximum quasi likelihood estimation methods have shown to give better estimates for $k$ (\cite{clark}).



In order to check the asymptotic normality of the MLE estimators, normal probability plots (QQ plots) for Model 3 parameters are provided (see Figure \ref{empqqplot}).  We see that the estimated and asymptotic densities of the model parameters $\nu$ are in agreement with each other.  

\section{Real Data Analysis}
Dengue is a life-threatening mosquito-borne viral infection  transmitted to humans through the bites of infected Aedes aegypti mosquitoes. The dengue virus comprises of four distinct but closely related serotypes (DEN 1-4). It is one of the most severe health problems being faced globally, with WHO reporting that there are 390 million dengue infections per year worldwide. In 2015, India alone reported 99913 dengue cases including 200 deaths.  
The dengue data under investigation consist of weekly dengue counts  in a three year period, January $2013$ to November $2015$, collected from  the  various divisions of an Indian metro city  by the public health department. 
However, the  distribution of disease counts across  divisions are quite uneven with some of them reporting excessive zero  weekly counts (a maximum of 50$\%$) in the span of three years as compared to the rest. We suspect that these zero disease counts are due to underreporting of dengue cases from the divisions. \cite{shepard2014} reports  a  prevalence of severe underreporting of dengue cases in India. Note the dengue data set is new and unpublished. 

\subsection{An Initial Look at the Dengue Data}
Taking a preliminary glance at the dengue data we note that the data set consists of weekly counts spanning over three years. 
The main part of the city is divided into several administrative divisions, which are numbered alphabetically from A-F. 
Few of these divisions (A-D) report a high quantity of zero dengue counts, while the rest report mostly  non-zero counts consistently throughout the year. We suspect that the zero dengue counts in divisions A-D arise due to their large distance from the national sentinel dengue surveillance hospitals, which leads to an under/non-reporting of active dengue cases. In this article, we analyze the dengue counts from division A, the histogram in  Figure \ref{data_hist} shows division A has a high percentage of zeros (66 in total, i.e.,  $44.29\%$).
 
The distribution of weekly  dengue counts (test results based on the MAC ELISA method) from division A  in the period January 2013 - November 2015 (149 weeks)  (see Figure \ref{data_hist}) shows an increasing trend and yearly seasonality with high dengue counts reported during monsoon months. This non-stationary behaviour of the dengue time series is also apparent from the ACF and PACF (see Figure \ref{acfdata}).

\subsection{Analyzing the Dengue Data}\label{real}
The proposed zero inflated  NB-ARMA model for various values of the pairs $(p_1,q_1)$ and $(p_2,q_2)$ is fitted to the dengue data. The NB mean $\lambda_t$ is modeled using a log model, while a logit model is fitted to  $\pi_t$. Due to the nonstationary behaviour detected in the  dengue count data (see Figure \ref{acfdata}), we tested for presence of trend and seasonality. A linear trend and yearly seasonality (represented by a pair of sine and cosine terms) were found to be statistically significant. The weather covariates introduced in the mean models were, average temperature ($T_{avg}$), relative humidity (HMD) and rainfall (RF). However, from preliminary analysis (not reported here) it was seen that $\pi_t$ was not affected significantly by any of the weather covariates or any order of AR or MA terms. Based on these observations, we considered the  following model forms for $W_t$ and $M_t$ (note that in each model $M_t$ has a simple form with only trend and seasonality terms):

\textbf{Model 1:}
\begin{equation*}\label{M1}
W_{t}=x^{T}_{t}{\beta}+\sum_{j=1}^{5}{\theta}_{j}e_{t-j},\qquad M_{t}=u^{T}_{t}{\delta},
\end{equation*}
where, $x_{t}^{T}=[\begin{matrix}1~t^{'}/(N-1)~{\cos}(2\pi t^{'}/52),{\sin}(2\pi t^{'}/52)~RF~HMD~T_{avg}\end{matrix}]$ and \\$u_{t}^{T}=[\begin{matrix}1~t^{'}/(N-1)~{\cos}(2\pi t^{'}/52)~{\sin}(2\pi t^{'}/52)\end{matrix}]$, here $t' := (t-1)$ to make the range of trend regressor ($t'/(N-1)$) lie in the unit interval [0,1].

\textbf{Model 2:}
\begin{equation*}\label{M2}
W_{t}=x^{T}_{t}{\beta}+\sum_{j=1}^{4}{\theta}_{j}e_{t-j},\qquad M_{t}=u^{T}_{t}{\delta},
\end{equation*}

\textbf{Model 3:}
\begin{equation*}\label{M3}
W_{t}=x^{T}_{t}{\beta}+\sum_{i=1}^{5}{\phi}_{i}Z_{t-i}+\sum_{j=1}^{5}{\theta}_{j}e_{t-i},\qquad M_{t}=u^{T}_{t}{\delta},
\end{equation*}

\textbf{Model 4:}
\begin{equation*}\label{M4}
W_{t}=x^{T}_{t}{\beta}+\sum_{j=1}^{3}{\theta}_{j}e_{t-j},\qquad M_{t}=u^{T}_{t}{\delta},
\end{equation*}
for models 2-4, $x_{t}^{T}=[\begin{matrix}1~t^{'}/(N-1)~{\cos}(2\pi t'/52)~{\sin}(2\pi t'/52)~HMD~T_{avg}\end{matrix}$] and $u_{t}^{T}=[\begin{matrix}1~t^{'}/(N-1)~{\cos}(2\pi t'/52)~{\sin}(2\pi t'/52)\end{matrix}$].
%

\textbf{Model 5:}
\begin{equation*}\label{M7}
W_{t}=x^{T}_{t}{\beta}+{\theta}_{1}e_{t-1}+{\theta}_{3}e_{t-3},\qquad M_{t}=u^{T}_{t}{\delta},
\end{equation*}

\textbf{Model 6:}
\begin{equation*}\label{M8}
W_{t}=x^{T}_{t}{\beta},\qquad M_{t}=u^{T}_{t}{\delta},
\end{equation*}
where $x_{t}^{T}=[\begin{matrix}1~t^{'}/(N-1)~{\cos}(2\pi t'/52)~{\sin}(2\pi t'/52)~HMD \end{matrix}]$, \\$u_{t}^{T}=[\begin{matrix}1~{\cos}(2\pi t'/52)~{\sin}(2\pi t'/52) \end{matrix}]$ and for all models (1-6) \\$e_{t}= (Y_{t}-\lambda_{t}(1-\pi_{t}))/\sqrt{ \Big[ \lambda_{t}(1-\pi_{t})(1+\lambda_{t}\pi_{t}+\frac{\lambda_{t}}{k})\Big]}$.

We choose M5 as the best model based on  values of AIC and BIC (see Table \ref{modelselectiontable}). The deviance and chi-squared values for M5 are also small when compared to models M1-M4 and M6. The  parameter estimates based on EM algorithm, standard errors and p-values for the chosen model M5 are given in Table \ref{tab:table5}. No remaining pattern is detected in the ACF and PACF plots of the  residuals from M5 (see Figure \ref{acfresiduals}), this is also supported by a high p-value ($>0.45$) of the Box-Ljung test (upto lag 10). The actual observed counts versus the fitted (conditional) mean $\Lambda_{t}$ plot from M5 is shown in Figure \ref{counts_fitted}. The randomized quantile residual plot  corresponding to M5  in Figure \ref{residual_qqplot},  show no evidence of any systematic pattern among the randomized quantile residuals and also QQ plot of the quantile residuals satisfy the normality assumptions. 

Further, the average probability of an excess zero is computed as \citep{Lambert1992},
\begin{align}
	\hat{p}_0 = \frac{\text{Number of}~(y_{t}=0) \text{ observed } -\sum_{y_{t}=0}^{}\Big(\frac{\hat{k}}{\hat{k}+\hat{\lambda}_{t}}\Big)^{\hat{k}}}{N}
\end{align}
From the observed dengue counts, the fraction of zeros was found to be 0.4429 (66 zeros out of 149 counts) and from model M5 (NB part), the fraction of zeros is estimated to be is 0.2833. Thus, the estimated average probability of an excess zero is $\hat{p}_0 = 0.1597$. 
It shows that the zero-inflated part of model M5 estimates these 15.97\% of excess zeros along with 28.33\% of zeros from the fitted NB model. 

Significant p-values (at level of significance $10\%$) obtained by running one-sided Vuong's test (\cite{vuong}) comparing M5 with various NB models (without zero-inflation), also confirms the need for fitting a zero inflated model to the dengue data. 
	
\subsection{A Comparative Study}
Statistical modeling of dengue data is not a new topic in the biomedical literature.  Some popular techniques  for analyzing dengue incidence include, generalized linear models such as Poisson and NB multivariate regression models \citep{xu2014,minh2014,lu2009}. Modelling strategies based on Gaussian assumptions to model count time data such as, ARIMA and SARIMA models with climate covariates \citep{luz2008,gharbi2011,wongkoon2012,martinez2011} are also popularly used to analyse dengue prevalence.  However, due to the non-negative and integer valued nature of the data, these ARIMA/SARIMA models may give undesirable results and general exponential family formulation may be required (\cite{cox}). Temporal models, like autoregressive (AR)  Poisson or NB distributions \citep{lu2009,briet2013} have also been fitted to dengue data.  However, most  of these  models  have a simple form with at most a first order AR term on the past responses and current values of the covariates. These simple model forms are shown  to be inadequate in capturing the complex disease trajectory over time in the following comparative study.

We compared the performance of M5 with some popular statistical dengue models available in the literature. The alternative models compared are zero inflated NB models with only AR terms, NB models with no zero inflation, ARIMA and SARIMA models. 
\begin{itemize}
\item Alternative Model (AM-1): Zero inflated Negative Binomial with AR(1) terms
\begin{align}
Y_t|s_t,H_t\sim \text{NB }(k,(1-s_t)\lambda_t) \qquad s_t|H_t\sim \text{Bernoulli }(\pi_t).\end{align} The means $\lambda_t$ and $\omega_t$ are modeled using the link functions,
\begin{align}
\log(\lambda_t)=x^{T}_{t}{\beta}+\theta I_{\{y_{t-1}>0\}},\qquad \text{logit }(\pi_t)=u^{T}_{t}{\delta},
\end{align}
where $H_t=(y_{t-1},x_t,u_t)$, \\$x_{t}^{T}=[\begin{matrix}1,t'/(N-1),\text{cos}(2\pi t'/52),\text{sin}(2\pi t'/52),HMD\end{matrix}$] and \\$u_{t}^{T}=[\begin{matrix}1,\text{cos}(2\pi t'/52),\text{sin}(2\pi t'/52) \end{matrix}$].

\item Alternative Model (AM-2): Negative Binomial $(k,\lambda_t/(\lambda_t+k))$ with AR(1) terms
\begin{align}\label{A36}
log(\lambda_t)=x_{t}^{T}\beta+ \theta\{\log(y^{*}_{t-j})-x_{t-j}^{T}\beta\},
\end{align}
where $y^{*}_{t-j}=\max(y_{t-j},c),\,0<c<1$ and \\$x_{t}^{T}=[\begin{matrix}1,t'/(N-1),\text{cos}(2\pi t'/52),\text{sin}(2\pi t'/52),HMD \end{matrix}$].

\item Alternative Model (AM-3): ARIMA $(p,r,q)$ model was fitted to the dengue counts. The best model with least AIC and BIC values was  the ARIMA (2,0,2).

\item Alternative Model (AM-4): The best SARIMA model with respect to AIC and BIC  values was  SARIMA $ (2,0,0)(1,0,1)^{52} $.
\end{itemize}

%
%
	

Table \ref{tab:denguedatacomparetable} shows that M5 has the lowest values of all four reported goodness of fit statistics. By setting different threshold values (i.e., if $\hat{Y}_t<\text{ threshold value}$, then assume $\hat{Y}_t=0$), we  conducted a sensitivity analysis (see results in Table \ref{ss} ) of zero counts prediction by model M5 and the zero inflated model AM-1, the table shows that the accuracy of M5 identifying the zeros is larger than AM-1.

%
%
%
%
Other than the above discussed alternative models, we also used  a neural network approach for dengue data fitting.  Neural networks with one/two hidden layers and ($1,3,7,10$)  hidden nodes with both sigmoid and tanh activation functions were tried. The inputs
 were the response variables till lag 4, humidity,  rainfall and average temperature.  The network with one hidden layer and node and the tanh activation function gave the lowest values of MSE $=16.47$ and MAD $= 4.05$. 
 
 \section{Analyzing the Syphilis Data}\label{real2}
 
 In this section, we fit our zero inflated NB-ARMA model to the time series of weekly syphilis disease counts in Virginia,
 USA between 2007 and 2010. The percentage of zeros in the syphilis data is $26.79\%$. 
 Our proposed model for various values of the pairs $(p_1,q_1)$ and $(p_2,q_2)$ is fitted to the syphilis data. Due to the nonstationary behaviour detected in the  syphilis count data we tested and found the presence of a significant linear trend. Based on a comparison of MSE, AIC and BIC values we chose the following model to fit the syphilis data:
  
%
 
 \begin{equation*}\label{M23}
 W_{t}=x^{T}_{t}{\beta}+\sum_{j=1}^{2}{\theta}_{j}e_{t-j},\qquad M_{t}=u^{T}_{t}{\delta},
 \end{equation*}
 
%
%
 where, $x_{t}^{T}=[\begin{matrix}1&t/1000 \end{matrix}] $ and $u_{t}^{T}=[\begin{matrix}1&t/1000 \end{matrix}]$ for all models (1-6) $e_{t}= (Y_{t}-\lambda_{t}(1-\pi_{t}))/\sqrt{ \Big[ \lambda_{t}(1-\pi_{t})(1+\lambda_{t}\pi_{t}+\frac{\lambda_{t}}{k})\Big]}$.
 
 Note we chose the linear trend regression to be of the form $x_{t}^{T}=[\begin{matrix}1&t/1000 \end{matrix}] $, since this form has been used before in the literature (see \cite{yang2013} and \cite{ghahramani2020}). The  EM based results for the chosen model are shown in Table \ref{tab:estimates}. 

We compared our model's performance with the distribution-free functional response (FRM) model  of \cite{ghahramani2020}, a zero inflated Poisson (ZIP), and a zero inflated NB (ZINB). The comparison is shown in Table \ref{Compare_ghahramani}. Note that the MSE, AIC and BIC values for our zero-inflated NB-ARMA model are lesser than the other models frequently used in the literature. 
 
 \section{Concluding Remarks}
 
 In this article, we proposed a zero inflated NB-ARMA model to fit correlated count time data with excessive zeros and variation. The proposed model was successfully fitted to weekly dengue and syphilis counts with zero inflation. In a detailed comparison of the chosen dengue and syphilis models with  various alternative models available in the literature we saw that the selected model had the lowest values of RMSE and MAD values. All computations  were done using the R and Matlab statistical softwares. The programs are available on request from the first author. 
 
 In future, we are interested in extending the proposed model to a multivariate time series setup using copulas. As an example, we may consider setting up a model for dengue counts over time simultaneously for several regions.

\section*{Appendix A: Proof of Theorem 2}\label{appA}  
  	We first prove asymptotic existence and consistency of maximum partial likelihood estimator $\hat{\nu}$. By Taylor expansion at true parameters $\nu$,
  	\begin{align}
  	PL(\tilde{\nu}) = PL(\nu) + (\tilde{\nu}-\nu)^{T}S_{N}(\nu)-\frac{1}{2}(\tilde{\nu}-\nu)^{T}\tilde{H}_{N}(\tilde{\tilde{\nu}})(\tilde{\nu}-\nu),
  	\end{align}
  	where $\tilde{\tilde{\nu}}$ lies between $\tilde{\nu}$ and $\nu$. Then,
  	\begin{align}\label{taylor}
  	PL(\tilde{\nu}) - PL(\nu) = (\tilde{\nu}-\nu)^{T}S_{N}(\nu)-\frac{1}{2}(\tilde{\nu}-\nu)^{T}\tilde{H}_{N}(\tilde{\tilde{\nu}})(\tilde{\nu}-\nu).
  	\end{align}
  	Let $O_{N}(\delta):=\{\tilde{\nu}:||F_{N}(\nu)^{T/2}(\tilde{\nu}-\nu)||\leq\delta\}$ and $\tilde{\lambda} := F_{N}(\nu)^{T/2}(\tilde{\nu}-\nu)/\delta$ where $F_{N}(\nu)^{T/2}:=(F_{N}(\nu)^{1/2})^{T}$. Then, $(\tilde{\nu}-\nu)^{T} = \tilde{\lambda}^{T}F_{N}(\nu)^{-1/2}\delta$, \eqref{taylor} becomes,
  	\begin{align}\label{}
  	PL(\tilde{\nu}) - PL(\nu) = \delta\tilde{\lambda}^{T}F_{N}(\nu)^{-1/2}S_{N}(\nu)-\frac{\delta^{2}}{2}\tilde{\lambda}^{T}F_{N}(\nu)^{-1/2}\tilde{H}_{N}(\tilde{\tilde{\nu}})F_{N}(\nu)^{-T/2}\tilde{\lambda}.
  	\end{align}
  	We are going to prove that for every $\eta>0$ there exist $N$ and $\delta$ such that
  	\begin{align}\label{prob}
  	P\left[(PL(\tilde{\nu}) - PL(\nu))<0~\forall\tilde{\nu}\in\partial O_{N}(\delta)\right]\geq 1-\eta,
  	\end{align} 
  	where $\partial O_{N}(\delta):=\{\tilde{\nu}:||F_{N}(\nu)^{T/2}(\tilde{\nu}-\nu)||=\delta\}$ is a boundary set of $O_{N}(\delta)$. This show that with probability tending to one, there exists a local maximum inside $O_{N}(\delta)$. Then the left part of inequality \eqref{prob} becomes,
  	\begin{align}\label{prob2}
  	P&\left[(PL(\tilde{\nu}) - PL(\nu))<0\right]\nonumber\\&=P\left[\delta\tilde{\lambda}^{T}F_{N}(\nu)^{-1/2}S_{N}(\nu)-\frac{\delta^{2}}{2}\tilde{\lambda}^{T}F_{N}(\nu)^{-1/2}\tilde{H}_{N}(\tilde{\tilde{\nu}})F_{N}(\nu)^{-T/2}\tilde{\lambda}<0\right]\nonumber\\
  	&=P\left[\delta\tilde{\lambda}^{T}F_{N}(\nu)^{-1/2}S_{N}(\nu)-\frac{\delta^{2}}{2}\tilde{\lambda}^{T}F_{N}(\nu)^{-1/2}\tilde{H}_{N}(\tilde{\tilde{\nu}})F_{N}(\nu)^{-T/2}\tilde{\lambda}<0\right]
  	\end{align} 
  	Note that $\tilde{\lambda}^{T}\tilde{\lambda}=1$. Consider
  	\begin{align}\label{ineq}
  	\delta\tilde{\lambda}^{T}F_{N}(\nu)^{-1/2}&S_{N}(\nu)-\frac{\delta^{2}}{2}\tilde{\lambda}^{T}F_{N}(\nu)^{-1/2}\tilde{H}_{N}(\tilde{\tilde{\nu}})F_{N}(\nu)^{-T/2}\tilde{\lambda}\nonumber\\&\leq\delta||\tilde{\lambda}F_{N}(\nu)^{-1/2}S_{N}(\nu)||-\frac{\delta^{2}}{2}\lambda_{min}(F_{N}(\nu)^{-1/2}\tilde{H}_{N}(\tilde{\tilde{\nu}})F_{N}(\nu)^{-T/2})\nonumber\\&\leq\delta||\tilde{\lambda}||||F_{N}(\nu)^{-1/2}S_{N}(\nu)||-\frac{\delta^{2}}{2}\lambda_{min}(F_{N}(\nu)^{-1/2}\tilde{H}_{N}(\tilde{\tilde{\nu}})F_{N}(\nu)^{-T/2})\nonumber\\
  	&=\delta||F_{N}(\nu)^{-1/2}S_{N}(\nu)||-\frac{\delta^{2}}{2}\lambda_{min}(F_{N}(\nu)^{-1/2}\tilde{H}_{N}(\tilde{\tilde{\nu}})F_{N}(\nu)^{-T/2})
  	\end{align}
  	Substitute \eqref{ineq} in \eqref{prob2}, then \eqref{prob2} becomes
  	\begin{align}\label{ineq2}
  	P&\left[(PL(\tilde{\nu}) - PL(\nu))<0\right]\nonumber\\&=P\left[||F_{N}(\nu)^{-1/2}S_{N}(\nu)||^{2}\leq\frac{\delta^{2}}{4}\lambda_{min}^{2}\big(F_{N}(\nu)^{-1/2}\tilde{H}_{N}(\tilde{\tilde{\nu}})F_{N}(\nu)^{-T/2}\big)\right]
  	\end{align} 
  	From Lemma \ref{continuity}, we know that $\frac{\tilde{H}_{N}(\nu)}{N}\xrightarrow{P}G(\nu)$ as $N\rightarrow\infty$. Then, the normalized information matrix $F_{N}(\nu)^{-1/2}\tilde{H}_{N}(\nu)F_{N}(\nu)^{-T/2}$ converges in probability to identity matrix $I$. Note that the set $\{O_{N}(\delta)\}$ shrinks to $\nu$ as $N\rightarrow\infty$. Thus, the normalized information matrix $F_{N}(\nu)^{-1/2}\tilde{H}_{N}(\tilde{\tilde{\nu}})F_{N}(\nu)^{-T/2}$ converges in probability to identity matrix $I_{}$ since $\tilde{\tilde{\nu}}$ goes to $\nu$ as $N\rightarrow\infty$. From Cramer-Wold Device, for any $\lambda\neq0$, $\lambda^{T}F_{N}(\nu)^{-1/2}\tilde{H}_{N}(\tilde{\tilde{\nu}})F_{N}(\nu)^{-T/2}\lambda$ converges in probability to 1. Then there exist $\lambda$ such that $\lambda_{min}(F_{N}(\nu)^{-1/2}\tilde{H}_{N}(\tilde{\tilde{\nu}})F_{N}(\nu)^{-T/2})$ converges in probability to 1. From Lemma 3.4 of \cite{fokianos1998}, the expression \eqref{ineq2} becomes
  	\begin{align}\label{ineq3}
  	P&\left[(PL(\tilde{\nu}) - PL(\nu))<0\right]\nonumber\\&=P\left[||F_{N}(\nu)^{-1/2}S_{N}(\nu)||^{2}\leq\frac{\delta^{2}}{4}\lambda_{min}^{2}\big(F_{N}(\nu)^{-1/2}\tilde{H}_{N}(\tilde{\tilde{\nu}})F_{N}(\nu)^{-T/2}\big)\right]\nonumber\\
  	&\geq P\left[||F_{N}(\nu)^{-1/2}S_{N}(\nu)||^{2}\leq\frac{\delta^{2}}{4}\right]\nonumber\\
  	&\geq 1-\frac{4}{\delta^{2}}E\left[||F_{N}(\nu)^{-1/2}S_{N}(\nu)||^{2}\right].
  	\end{align} 
  	For a given $\eta>0$, choosing $\delta = \sqrt{\frac{4}{\eta}E\left[||F_{N}(\nu)^{-1/2}S_{N}(\nu)||^{2}\right]}$, then we have
  	\begin{align}\label{ineq4}
  	P\left[(PL(\tilde{\nu}) - PL(\nu))<0 \forall\tilde{\nu}\in\partial O_{N}(\delta)\right]\geq 1-\eta.
  	\end{align}
  	Therefore asymptotic existence is established. More specifically, there exists a sequence of maximum partial likelihood estimators $\hat{\nu}$, such that for any $\eta>0$,
  	there are $\delta$ and $N_{1}$ such that 
  	\begin{align}\label{ineq5}
  	P\left[\hat{\nu}\in O_{N}(\delta)\right]&\geq 1-\eta\nonumber\\
  	P\left[||F_{N}(\nu)^{-T/2}(\hat{\nu}-\nu)||\leq\delta\right]&\geq 1-\eta\nonumber\\
  	P\left[||\hat{\nu}-\nu||\leq\frac{\delta}{\lambda_{min}(F_{N}(\nu))}\right]&\geq 1-\eta
  	\end{align}
  	Thus, the maximum partial likelihood estimator $\hat{\nu}$ is locally unique. Therefore consistency is established. 
  	
  	Next, we prove asymptotic normality. By Taylor expansion around $\hat{\nu}$, and using the mean value theorem for multivariate function we obtain for $0\leq s\leq 1$,
  	\begin{align}\label{asymt}
  	S_{N}(\nu) =& S_{N}(\hat{\nu}) + \tilde{H}_{N}(\nu+s(\hat{\nu}-\nu))(\hat{\nu}-\nu)\nonumber\\
  	=& \tilde{H}_{N}(\nu+s(\hat{\nu}-\nu))(\hat{\nu}-\nu),
  	\end{align}
  	where $\nu+s(\hat{\nu}-\nu)$ lies between $\hat{\nu}$ and $\nu$. Multiply $F_{N}(\nu)^{-1/2}$ both sides to \eqref{asymt}, then
  	\begin{align}\label{asymt2}
  	F_{N}(\nu)^{-1/2}S_{N}&(\nu) = F_{N}(\nu)^{-1/2}\tilde{H}_{N}(\nu+s(\hat{\nu}-\nu))(\hat{\nu}-\nu)\nonumber\\
  	&= \big(F_{N}(\nu)^{-1/2}\tilde{H}_{N}(\nu+s(\hat{\nu}-\nu))F_{N}(\nu)^{-T/2}\big)\big(F_{N}(\nu)^{T/2}(\hat{\nu}-\nu)\big).
  	\end{align}
  	From Theorem \ref{score_MDS}, The score function $S_{N}(\nu)$ converges in distribution to $\mathcal{N}(0, G(\nu))$, as $N\rightarrow\infty$. The normalized score function $F_{N}^{-1/2}S_{N}(\nu)$ converges in distribution to $\mathcal{N}(0, I)$. From consistency \eqref{ineq5} and Lemma \ref{continuity}, $F_{N}(\nu)^{-T/2}\tilde{H}_{N}(\nu+s(\hat{\nu}-\nu))F_{N}(\nu)^{-T/2}$ converges in probability to identity matrix $I_{}$. Then, the expression \eqref{asymt2} becomes,
  	\begin{align}\label{asymt3}
  	F_{N}(\nu)^{T/2}(\hat{\nu}-\nu)\rightarrow \mathcal{N}(0,I),
  	\end{align}
  	in distribution as $N\rightarrow\infty$. The normalized conditional information matrix \\$F_{N}(\nu)^{-1/2}G_{N}(\nu)F_{N}(\nu)^{-1/2}$ converges in probability to identity matrix $I_{}$. Choosing $G_{N}(\nu)^{1/2}$ such that $F_{N}(\nu)^{-1/2}G_{N}(\nu)^{1/2}$ is the Cholesky square root of \\$F_{N}(\nu)^{-1/2}G_{N}(\nu)F_{N}(\nu)^{-1/2}$. We have from the continuity of the square root that $F_{N}(\nu)^{-1/2}G_{N}(\nu)^{1/2}$ converges in probability to identity matrix $I_{}$. Thus,
  	\begin{align}\label{asymt4}
  	G_{N}(\nu)^{T/2}(\hat{\nu}-\nu)= G_{N}(\nu)^{T/2}F_{N}(\nu)^{-T/2}F_{N}(\nu)^{T/2}(\hat{\nu}-\nu) \rightarrow \mathcal{N}(0,I)
  	\end{align}
  	as $N\rightarrow\infty$. From the continuity of the square root $\frac{G_{N}(\nu)^{T/2}}{\sqrt{N}}$ converges in probability $G(\nu)^{T/2}$. Then \eqref{asymt4} becomes,
  	\begin{align}\label{}
  	\sqrt{N}\Big(\frac{G_{N}(\nu)^{T/2}}{\sqrt{N}}\Big)(\hat{\nu}-\nu) \rightarrow \mathcal{N}(0,I),
  	\end{align}
  	\begin{align}\label{asymt5}
  	\sqrt{N}(\hat{\nu}-\nu) \rightarrow \mathcal{N}(0,G(\nu)^{-1}),
  	\end{align}
  	in distribution as $N\rightarrow\infty$.
  
  \section*{Acknowledgements}
The work of S. Mukhopadhyay was supported by the Science and Research Engineering Board (Department of Science
and Technology, Government of India) [File Number: EMR/2016/005142] and Wadhwani Research Centre for Bio-Engineering. We would like to acknowledge Professors Konstantinos Fokianos and Monika Bhattacharjee for their helpful suggestions. 
 \bibliographystyle{chicago}
 \bibliography{References}

\begin{thebibliography}{}

\bibitem[\protect\citeauthoryear{Benjamin, Rigby, and Stasinopoulos}{Benjamin
  et~al.}{2003}]{Benjamin2003}
Benjamin, M.~A., R.~A. Rigby, and D.~M. Stasinopoulos (2003).
\newblock Generalized autoregressive moving average models.
\newblock {\em Journal of the American Statistical Association\/}~{\em
  98\/}(461), 214--223.

\bibitem[\protect\citeauthoryear{Bri{\"e}t, Amerasinghe, and
  Vounatsou}{Bri{\"e}t et~al.}{2013}]{briet2013}
Bri{\"e}t, O.~J., P.~H. Amerasinghe, and P.~Vounatsou (2013).
\newblock Generalized seasonal autoregressive integrated moving average models
  for count data with application to malaria time series with low case numbers.
\newblock {\em PLoS One\/}~{\em 8\/}(6), e65761.

\bibitem[\protect\citeauthoryear{Chan and Ledolter}{Chan and
  Ledolter}{1995}]{10.2307/2291149}
Chan, K.~S. and J.~Ledolter (1995).
\newblock Monte {Carlo} {EM} estimation for time series models involving
  counts.
\newblock {\em Journal of the American Statistical Association\/}~{\em 90},
  242--252.

\bibitem[\protect\citeauthoryear{Chiogna and Gaetan}{Chiogna and
  Gaetan}{2002}]{chiogna}
Chiogna, M. and C.~Gaetan (2002).
\newblock Dynamic generalized linear models with application to environmental
  epidemiology.
\newblock {\em Applied Statistics\/}~{\em 51}, 453--468.

\bibitem[\protect\citeauthoryear{Clark and Perry}{Clark and
  Perry}{1989}]{clark}
Clark, S. and J.~Perry (1989).
\newblock Estimation of the negative binomial parameter k by maximum
  quasilikelihood.
\newblock {\em Biometrics\/}~{\em 45\/}(1), 309--316.

\bibitem[\protect\citeauthoryear{Cox}{Cox}{1981}]{cox}
Cox, D.~R. (1981).
\newblock Statistical analysis of time series: Some recent developments.
\newblock {\em Scandinavian Journal of Statistics\/}~{\em 8}, 93--115.

\bibitem[\protect\citeauthoryear{Davidson}{Davidson}{1994}]{davidson1994}
Davidson, J. (1994).
\newblock {\em Stochastic limit theory: An introduction for econometricians}.
\newblock OUP Oxford.

\bibitem[\protect\citeauthoryear{Davis, Dunsmuir, and Streett}{Davis
  et~al.}{2003}]{Davis2003}
Davis, R.~A., W.~T.~M. Dunsmuir, and S.~B. Streett (2003).
\newblock Observation-driven models for {P}oisson counts.
\newblock {\em Biometrika\/}~{\em 90\/}(4), 777--790.

\bibitem[\protect\citeauthoryear{Davis and Liu}{Davis and
  Liu}{2016}]{davis2016}
Davis, R.~A. and H.~Liu (2016).
\newblock Theory and inference for a class of observation-driven models with
  application to time series of counts.
\newblock {\em Statistica Sinica\/}~{\em 102}, 1673--1707.

\bibitem[\protect\citeauthoryear{Dempster, Laird, and Rubin}{Dempster
  et~al.}{1977}]{dempster}
Dempster, A.~P., N.~M. Laird, and D.~B. Rubin (1977).
\newblock Maximum likelihood from incomplete data via the {EM} algorithm.
\newblock {\em Journal of the Royal Statistical Society. Series B\/}~{\em
  39\/}(1), 1--38.

\bibitem[\protect\citeauthoryear{Douc, Doukhan, and Moulines}{Douc
  et~al.}{2013}]{douc2013}
Douc, R., P.~Doukhan, and E.~Moulines (2013).
\newblock Ergodicity of observation-driven time series models and consistency
  of the maximum likelihood estimator.
\newblock {\em Stochastic Processes and their Applications\/}~{\em 123}, 2620--
  2647.

\bibitem[\protect\citeauthoryear{Durbin and Koopman}{Durbin and
  Koopman}{1997}]{durbin1997}
Durbin, J. and S.~J. Koopman (1997).
\newblock Monte {Carlo} maximum likelihood estimation for non-gaussian state
  space models.
\newblock {\em Biometrika\/}~{\em 84}, 669--684.

\bibitem[\protect\citeauthoryear{Durbin and Koopman}{Durbin and
  Koopman}{2000}]{durbin2000}
Durbin, J. and S.~J. Koopman (2000).
\newblock Time series analysis of non-{G}aussian observations based on state
  space models from both classical and {B}ayesian perspectives.
\newblock {\em Journal of the Royal Statistical Society. Series B\/}~{\em
  62\/}(1), 3--56.
\newblock With discussion and a reply by the authors.

\bibitem[\protect\citeauthoryear{Durbin and Koopman}{Durbin and
  Koopman}{2012}]{dk}
Durbin, J. and S.~J. Koopman (2012).
\newblock {\em Time Series Analysis by State Space Methods}.
\newblock Oxford University Press.

\bibitem[\protect\citeauthoryear{Fahrmeir and Tutz}{Fahrmeir and
  Tutz}{2001}]{fahrmeir}
Fahrmeir, L. and G.~Tutz (2001).
\newblock {\em Multivariate Statistical Modeling Based on Generalized Linear
  Models}.
\newblock Springer-Verlag, New York.

\bibitem[\protect\citeauthoryear{Fahrmeir and Wagenpfeil}{Fahrmeir and
  Wagenpfeil}{1997}]{fahrwagen}
Fahrmeir, L. and S.~Wagenpfeil (1997).
\newblock Penalized likelihood estimation and iterative {Kalman} filtering for
  non-gaussian dynamic regression models.
\newblock {\em Computational Statistics and Data Analysis\/}~{\em 24},
  295--320.

\bibitem[\protect\citeauthoryear{Fokianos and Kedem}{Fokianos and
  Kedem}{1998}]{fokianos1998}
Fokianos, K. and B.~Kedem (1998).
\newblock Prediction and classification of non-stationary categorical time
  series.
\newblock {\em Journal of Multivariate Analysis\/}~{\em 67}, 277--296.

\bibitem[\protect\citeauthoryear{Fokianos and Kedem}{Fokianos and
  Kedem}{2004}]{fokianos2003}
Fokianos, K. and B.~Kedem (2004).
\newblock Partial likelihood inference for time series following generalized
  linear models.
\newblock {\em Journal of Time Series Analysis\/}~{\em 25\/}(2), 173--197.

\bibitem[\protect\citeauthoryear{Fokianos, Rahbek, and Tj{\o}stheim}{Fokianos
  et~al.}{2009a}]{fokianos2009}
Fokianos, K., A.~Rahbek, and D.~Tj{\o}stheim (2009a).
\newblock Poisson autoregression.
\newblock {\em Journal of the American Statistical Association\/}~{\em 104},
  1430--1439.

\bibitem[\protect\citeauthoryear{Fokianos, Rahbek, and Tj{\o}stheim}{Fokianos
  et~al.}{2009b}]{fokianos2009poisson}
Fokianos, K., A.~Rahbek, and D.~Tj{\o}stheim (2009b).
\newblock Poisson autoregression.
\newblock {\em Journal of the American Statistical Association\/}~{\em
  104\/}(488), 1430--1439.

\bibitem[\protect\citeauthoryear{Fokianos and Tj{\o}stheim}{Fokianos and
  Tj{\o}stheim}{2011}]{fokianos2011}
Fokianos, K. and D.~Tj{\o}stheim (2011).
\newblock Log–linear poisson autoregression.
\newblock {\em Journal of Multivariate Analysis\/}~{\em 102}, 563--578.

\bibitem[\protect\citeauthoryear{Gamerman}{Gamerman}{1998}]{gamerman1998}
Gamerman, D. (1998).
\newblock Markov chain monte carlo for dynamic generalised linear models.
\newblock {\em Biometrika\/}~{\em 85}, 215--227.

\bibitem[\protect\citeauthoryear{Gamerman, Santos, and Franco}{Gamerman
  et~al.}{2013}]{gamerman2013}
Gamerman, D., T.~R. Santos, and G.~C. Franco (2013).
\newblock A non‐gaussian family of state‐space models with exact marginal
  likelihood.
\newblock {\em Journal of Time Series Analysis\/}~{\em 34}, 625--645.

\bibitem[\protect\citeauthoryear{Ghahramani and White}{Ghahramani and
  White}{2020}]{ghahramani2020}
Ghahramani, M. and S.~White (2020).
\newblock Time series regression for zero-inflated and overdispersed count
  data: A functional response model approach.
\newblock {\em Journal of Statistical Theory and Practice\/}~{\em 14\/}(2),
  1--18.

\bibitem[\protect\citeauthoryear{Gharbi, Quenel, Gustave, Cassadou, La~Ruche,
  Girdary, and Marrama}{Gharbi et~al.}{2011}]{gharbi2011}
Gharbi, M., P.~Quenel, J.~Gustave, S.~Cassadou, G.~La~Ruche, L.~Girdary, and
  L.~Marrama (2011).
\newblock Time series analysis of dengue incidence in {Guadeloupe, French West
  Indies}: forecasting models using climate variables as predictors.
\newblock {\em BMC Infectious Diseases\/}~{\em 11\/}(1), 166.

\bibitem[\protect\citeauthoryear{Godolphin and Triantafyllopoulos}{Godolphin
  and Triantafyllopoulos}{2006}]{godolphin}
Godolphin, E.~J. and K.~Triantafyllopoulos (2006).
\newblock Decomposition of time series models in state-space form.
\newblock {\em Computational Statistics and Data Analysis\/}~{\em 50},
  2232--2246.

\bibitem[\protect\citeauthoryear{Hall and Heyde}{Hall and
  Heyde}{2014}]{hall2014martingale}
Hall, P. and C.~C. Heyde (2014).
\newblock {\em Martingale limit theory and its application}.
\newblock Academic press, London.

\bibitem[\protect\citeauthoryear{Hardy, Littlewood, and Polya}{Hardy
  et~al.}{1934}]{hardy1934inequalities}
Hardy, G., J.~Littlewood, and Polya (1934).
\newblock {\em Inequalities}.
\newblock Cambridge University Press.

\bibitem[\protect\citeauthoryear{Jain, Sontisirikit, Iamsirithaworn, and
  Prendinger}{Jain et~al.}{2019}]{jain2019}
Jain, R., S.~Sontisirikit, S.~Iamsirithaworn, and H.~Prendinger (2019).
\newblock Prediction of dengue outbreaks based on disease surveillance,
  meteorological and socio-economic data.
\newblock {\em BMC Infectious Diseases\/}~{\em 19\/}(1), 272.

\bibitem[\protect\citeauthoryear{Jazi, Jones, and Lai}{Jazi
  et~al.}{2012}]{jazi2012}
Jazi, M.~A., G.~Jones, and C.-D. Lai (2012).
\newblock First-order integer valued {AR} processes with zero inflated poisson
  innovations.
\newblock {\em Journal of Time Series Analysis\/}~{\em 33\/}(6), 954--963.

\bibitem[\protect\citeauthoryear{Jensen and Rahbek}{Jensen and
  Rahbek}{2007}]{jensen2007law}
Jensen, S.~T. and A.~Rahbek (2007).
\newblock On the law of large numbers for (geometrically) ergodic markov
  chains.
\newblock {\em Econometric Theory\/}, 761--766.

\bibitem[\protect\citeauthoryear{Kedem and Fokianos}{Kedem and
  Fokianos}{2002}]{Kedem2002}
Kedem, B. and K.~Fokianos (2002).
\newblock {\em Regression models for time series analysis}.
\newblock John Wiley \& Sons, Hoboken, New Jersey.

\bibitem[\protect\citeauthoryear{Lambert}{Lambert}{1992}]{Lambert1992}
Lambert, D. (1992).
\newblock Zero-inflated poisson regression, with an application to defects in
  manufacturing.
\newblock {\em Technometrics\/}~{\em 34\/}(1), 1--14.

\bibitem[\protect\citeauthoryear{Li}{Li}{1994}]{Li1994}
Li, W.~K. (1994).
\newblock Time series models based on generalized linear models: Some further
  results.
\newblock {\em Biometrics\/}~{\em 50\/}(2), 506--511.

\bibitem[\protect\citeauthoryear{Lu, Lin, Tian, Yang, Sun, and Liu}{Lu
  et~al.}{2009}]{lu2009}
Lu, L., H.~Lin, L.~Tian, W.~Yang, J.~Sun, and Q.~Liu (2009).
\newblock Time series analysis of dengue fever and weather in {Guangzhou,
  China}.
\newblock {\em BMC Public Health\/}~{\em 9\/}(1), 395.

\bibitem[\protect\citeauthoryear{Luz, Mendes, Code{\c{c}}o, Struchiner, and
  Galvani}{Luz et~al.}{2008}]{luz2008}
Luz, P.~M., B.~V. Mendes, C.~T. Code{\c{c}}o, C.~J. Struchiner, and A.~P.
  Galvani (2008).
\newblock Time series analysis of dengue incidence in {Rio de Janeiro, Brazil}.
\newblock {\em The American Journal of Tropical Medicine and Hygiene\/}~{\em
  79\/}(6), 933--939.

\bibitem[\protect\citeauthoryear{Martinez and Silva}{Martinez and
  Silva}{2011}]{martinez2011}
Martinez, E.~Z. and E.~A. S.~d. Silva (2011).
\newblock Predicting the number of cases of dengue infection in {Ribeir{\~a}o
  Preto}, {S{\~a}o Paulo State, Brazil}, using a sarima model.
\newblock {\em Cadernos de Saude Publica\/}~{\em 27}, 1809--1818.

\bibitem[\protect\citeauthoryear{Minh~An and Rockl{\"o}v}{Minh~An and
  Rockl{\"o}v}{2014}]{minh2014}
Minh~An, D.~T. and J.~Rockl{\"o}v (2014).
\newblock Epidemiology of dengue fever in {Hanoi} from 2002 to 2010 and its
  meteorological determinants.
\newblock {\em Global Health Action\/}~{\em 7\/}(1), 23074.

\bibitem[\protect\citeauthoryear{Perumean-Chaney, Morgan, McDowall, and
  Aban}{Perumean-Chaney et~al.}{2013}]{chaney}
Perumean-Chaney, S.~E., C.~Morgan, D.~McDowall, and I.~Aban (2013).
\newblock Zero-inflated and overdispersed: what's one to do?
\newblock {\em Journal of Statistical Computation and Simulation\/}~{\em 83},
  1671--1683.

\bibitem[\protect\citeauthoryear{Piegorsch}{Piegorsch}{1990}]{piegorsch}
Piegorsch, W. (1990).
\newblock Maximum likelihood estimation for the negative binomial dispersion
  parameter.
\newblock {\em Biometrics\/}~{\em 46\/}(3), 863--867.

\bibitem[\protect\citeauthoryear{Qi, Li, and Zhu}{Qi et~al.}{2019}]{qi2019}
Qi, X., Q.~Li, and F.~Zhu (2019).
\newblock Modeling time series of count with excess zeros and ones based on
  {INAR (1)} model with zero-and-one inflated poisson innovations.
\newblock {\em Journal of Computational and Applied Mathematics\/}~{\em 346},
  572--590.

\bibitem[\protect\citeauthoryear{Schmidt and Pereira}{Schmidt and
  Pereira}{2011}]{schmidt2011}
Schmidt, A.~M. and J.~B.~M. Pereira (2011).
\newblock Modelling time series of counts in epidemiology.
\newblock {\em International Statistical Review\/}~{\em 79\/}(1), 48--69.

\bibitem[\protect\citeauthoryear{Shepard, Halasa, Tyagi, Adhish, Nandan,
  Karthiga, Chellaswamy, Gaba, Arora, Group, et~al.}{Shepard
  et~al.}{2014}]{shepard2014}
Shepard, D.~S., Y.~A. Halasa, B.~K. Tyagi, S.~V. Adhish, D.~Nandan,
  K.~Karthiga, V.~Chellaswamy, M.~Gaba, N.~K. Arora, I.~S. Group, et~al.
  (2014).
\newblock Economic and disease burden of dengue illness in india.
\newblock {\em The American Journal of Tropical Medicine and Hygiene\/}~{\em
  91\/}(6), 1235--1242.

\bibitem[\protect\citeauthoryear{Shephard}{Shephard}{1995}]{Shephard1995}
Shephard, N. (1995).
\newblock Generalized linear autoregressions.
\newblock Economics Papers~8., Economics Group, Nuffield College, University of
  Oxford.

\bibitem[\protect\citeauthoryear{Shephard and Pitt}{Shephard and
  Pitt}{1997}]{sp}
Shephard, N. and M.~K. Pitt (1997).
\newblock Likelihood analysis of non-gaussian measurement time series.
\newblock {\em Biometrika\/}~{\em 84}, 653--667.

\bibitem[\protect\citeauthoryear{Siriyasatien, Phumee, Ongruk, Jampachaisri,
  and Kesorn}{Siriyasatien et~al.}{2016}]{siriyasatien2016}
Siriyasatien, P., A.~Phumee, P.~Ongruk, K.~Jampachaisri, and K.~Kesorn (2016).
\newblock Analysis of significant factors for dengue fever incidence
  prediction.
\newblock {\em BMC Bioinformatics\/}~{\em 17\/}(1), 166.

\bibitem[\protect\citeauthoryear{Vuong}{Vuong}{1989}]{vuong}
Vuong, Q. (1989).
\newblock Likelihood ratio tests for model selection and non-nested hypotheses.
\newblock {\em Econometrica\/}~{\em 57}, 307--333.

\bibitem[\protect\citeauthoryear{West, Harrison, and Migon}{West
  et~al.}{1985}]{west}
West, M., P.~J. Harrison, and H.~S. Migon (1985).
\newblock Dynamic generalized linear models and {Bayesian} forecasting.
\newblock {\em Journal of the American Statistical Association\/}~{\em 80},
  73--96.

\bibitem[\protect\citeauthoryear{Wongkoon, Jaroensutasinee, and
  Jaroensutasinee}{Wongkoon et~al.}{2012}]{wongkoon2012}
Wongkoon, S., M.~Jaroensutasinee, and K.~Jaroensutasinee (2012).
\newblock Development of temporal modeling for prediction of dengue infection
  in northeastern {Thailand}.
\newblock {\em Asian Pacific Journal of Tropical Medicine\/}~{\em 5\/}(3),
  249--252.

\bibitem[\protect\citeauthoryear{Xu, Fu, Lee, Ma, Goh, Wong, Habibullah, Lee,
  Lim, Tambyah, et~al.}{Xu et~al.}{2014}]{xu2014}
Xu, H.-Y., X.~Fu, L.~K.~H. Lee, S.~Ma, K.~T. Goh, J.~Wong, M.~S. Habibullah,
  G.~K.~K. Lee, T.~K. Lim, P.~A. Tambyah, et~al. (2014).
\newblock Statistical modeling reveals the effect of absolute humidity on
  dengue in {Singapore}.
\newblock {\em PLoS Neglected Tropical Diseases\/}~{\em 8\/}(5), e2805.

\bibitem[\protect\citeauthoryear{Xu, Chen, Chen, Lin, et~al.}{Xu
  et~al.}{2020}]{xu2020adaptive}
Xu, X., Y.~Chen, C.~W. Chen, X.~Lin, et~al. (2020).
\newblock Adaptive log-linear zero-inflated generalized poisson autoregressive
  model with applications to crime counts.
\newblock {\em Annals of Applied Statistics\/}~{\em 14\/}(3), 1493--1515.

\bibitem[\protect\citeauthoryear{Yang, Zamba, and Cavanaugh}{Yang
  et~al.}{2013}]{yang2013}
Yang, M., G.~K. Zamba, and J.~E. Cavanaugh (2013).
\newblock Markov regression models for count time series with excess zeros: A
  partial likelihood approach.
\newblock {\em Statistical Methodology\/}~{\em 14}, 26--38.

\bibitem[\protect\citeauthoryear{Yau, Lee, and Carrivick}{Yau
  et~al.}{2004}]{yau2004}
Yau, K.~K., A.~H. Lee, and P.~J. Carrivick (2004).
\newblock Modeling zero-inflated count series with application to occupational
  health.
\newblock {\em Computer Methods and Programs in Biomedicine\/}~{\em 74\/}(1),
  47--52.

\bibitem[\protect\citeauthoryear{Zeger and Qaqish}{Zeger and
  Qaqish}{1988}]{ZegerQaqish1988}
Zeger, S.~L. and B.~Qaqish (1988).
\newblock Markov regression models for time series: a quasi-likelihood
  approach.
\newblock {\em Biometrics. Journal of the Biometric Society\/}~{\em 44\/}(4),
  1019--1031.

\bibitem[\protect\citeauthoryear{Zhu}{Zhu}{2012}]{zhu2012}
Zhu, F. (2012).
\newblock Zero-inflated poisson and negative binomial integer-valued {GARCH}
  models.
\newblock {\em Journal of Statistical Planning and Inference\/}~{\em 142\/}(4),
  826--839.

\end{thebibliography}
 \newpage

\begin{table}[h!]
	\tiny
	\begin{center}
		\caption{Simulation results for Model 3 in Section \ref{sim}}
		\label{SimCompTableMAreal}
		\begin{tabular}{|c|c|c|c|c|c|} 
			\hline
			\multicolumn{6}{|c|}{EM} \\
			\hline
			\multicolumn{6}{|c|}{N=30}\\ 
			\hline
			
			Parameters&True& Est.& S.E.&$|\text{Bias}|$ & C.I.  \\\hline
			Intercept&0.3  &  0.3461 &   0.0212  & 0.0461 &(0.3045, 0.3877) \\
			$t/30$& 0.0001 & 0.00045  &  0.0015  &0.00035 &(-0.0025, 0.0034) \\
			cos(2$\pi t$/52)&0.2  &  0.2103  &  0.0181  &0.0103 & (0.1748, 0.2458)  \\
			sin(2$\pi t$/52)&-0.4000  & -0.4253  &  0.0192  &0.0253& (-0.4629, -0.3877)  \\
			MA1&-3 &  -3.4125  &  0.0531  &0.4125& (-3.5166, -3.3084) \\
			MA3&-2  & -1.7526 &   0.0203 & 0.2474 &(-1.7924, -1.7128) \\
			Intercept&0.1 & 0.1811  &  0.0183  & 0.0811 &(0.1452, 0.2170)  \\
			cos(2$\pi t$/52)&-0.4 &  -0.4897  &  0.0216  &0.0897 &(-0.5320, -0.4474) \\
			sin(2$\pi t$/52)&-0.5 &  -0.5612  &  0.0195  &0.0612 &(-0.5994, -0.5230) \\
			Overdisp&2  &  2.8120  &  0.0636  & 0.8120 &(2.6873, 2.9367) \\
			\hline
			\multicolumn{6}{|c|}{N=100}\\ 
			\hline
			
			Intercept&0.3  &  0.2562  &  0.0180   & 0.0438&(0.2209, 0.2914) \\
			$t/100$& 0.0001 &   0.00027   & 0.00039 & 0.00017 &(-0.00049, 0.0010) \\
			cos(2$\pi t$/52)&0.2  &  0.2015  &  0.0096   &0.0015 &(0.1827, 0.2202)  \\
			sin(2$\pi t$/52)&-0.4000  &  -0.4016  &  0.0097 &0.0016 & (-0.4207, -0.3826) \\
			MA1&-3 &  -3.1471  &  0.0220 & 0.1471 &(-3.1902, -3.1040) \\
			MA3&-2  &   -2.1293  &  0.0169  &0.1293 &(-2.1624, -2.0961) \\
			Intercept&0.1 &0.1331  &  0.0100  &0.0331& (0.1135, 0.1527)\\
			cos(2$\pi t$/52)&-0.4 &   -0.4748  &  0.0146& 0.0748 & (-0.5034, -0.4463) \\
			sin(2$\pi t$/52)&-0.5 &  -0.5342  &  0.0136& 0.0342 & (-0.5609, -0.5075) \\
			Overdisp&2  &  2.6211  &  0.0467  &0.6211 & (2.5296, 2.7125)  \\
			
			\hline
			\multicolumn{6}{|c|}{N=500}\\ 
			\hline
			
			Intercept&0.3  &  0.2851  &  0.0058  & 0.0149 &(0.2737, 0.2965)\\
			$t/500$& 0.0001 &   0.000101  &  $1.8\times 10^{-06}$  & $10^{-06}$ &(0.00009, 0.000104)\\
			cos(2$\pi t$/52)&0.2  &   0.1966  &  0.0033  & 0.0034 &(0.1902, 0.2030)\\
			sin(2$\pi t$/52)&-0.4000  & -0.3983  &  0.0033  & 0.0017&(-0.4048, -0.3919)\\
			MA1&-3 &   -3.0117  &  0.0073 & 0.0117& (-3.0259, -2.9974)\\
			MA3&-2  &  -2.0098  &  0.0055  &0.0098 &(-2.0206, -1.9990)\\
			Intercept&0.1 &    0.0918  &  0.0034  &0.0082 & (0.0851, 0.0985)\\
			cos(2$\pi t$/52)&-0.4 &   -0.4125  &  0.0048 & 0.0125 &(-0.4220, -0.4030)\\
			sin(2$\pi t$/52)&-0.5 &  -0.4993  &  0.0049  &0.0007 &(-0.5089, -0.4897)\\
			Overdisp&2  &  2.1963  &  0.0124  & 0.1963 &(2.1720, 2.2205)\\
			
			\hline 
			\multicolumn{6}{|c|}{MLE}\\
			\hline
			\multicolumn{6}{|c|}{N=30}\\ 
			\hline
			Intercept&0.3   & 0.3917  &  0.0288  & 0.0917 &(0.3353, 0.4481)  \\	
			$t/30$&0.0001 &   0.00034  &  0.0027 & 0.00024 &(-0.0050, 0.0056) \\ 
			cos(2$\pi t$/52)&0.2  &  0.2813  &  0.0191  & 0.0813 &(0.2439, 0.3187) \\
			sin(2$\pi t$/52)&-0.4  &  -0.4513  &  0.0167& 0.0513 & (-0.4840, -0.4186) \\
			MA1&-3  &  -3.6412 &   0.0435 &0.6412  &(-3.7265, -3.5559) \\ 
			MA3&-2  &  -1.6513  &  0.0211 & 0.3487 &(-1.6927, -1.6099) \\
			Intercept&0.1  &  0.1884  &  0.0191  & 0.0884 &(0.1510, 0.2258) \\
			cos(2$\pi t$/52)&-0.4  &  -0.4981  &  0.0254 &0.0981 &(-0.5479, -0.4483) \\
			sin(2$\pi t$/52)&-0.5  &  -0.5815  &  0.0199 & 0.0815 &(-0.6205, -0.5425) \\
			Overdisp&2 &   2.9189   & 0.0731 &   &(2.7756, 3.0621) \\
			
			\hline
			\multicolumn{6}{|c|}{N=100}\\ 
			\hline
			Intercept&0.3   &   0.3617  &  0.0197  & 0.0617 &(0.3231, 0.4003)  \\
			$t/100$&0.0001 &     0.0002  &  0.0003   &0.0001 &(-0.00038, 0.00078)  \\
			cos(2$\pi t$/52)&0.2  &  0.2355  &  0.0100 &0.0355 &(0.2159, 0.2551)  \\
			sin(2$\pi t$/52)&-0.4  &  -0.4060  &  0.0102  & 0.0060 &(-0.4259, -3.860) \\
			MA1&-3  &   -3.1803  &  0.0240 &  0.1803 &(-3.2273, -3.1332)  \\
			MA3&-2  &   -2.1311  &  0.0184  & 0.1311 &(-2.1673, -2.0950) \\
			Intercept&0.1  & 0.1640  &  0.0185  & 0.0640 &(0.1277, 0.2002) \\
			cos(2$\pi t$/52)&-0.4  &  -0.4538  &  0.0199  & 0.0538 &(-0.4928, -0.4147) \\
			sin(2$\pi t$/52)&-0.5  & -0.5591  &  0.0222 & 0.0591 & (-0.6026, -0.5155) \\
			Overdisp&2 &  2.5507  &  0.0500  & 0.5507 &(2.4527, 2.6486)  \\
			
			\hline
			\multicolumn{6}{|c|}{N=500}\\ 
			\hline
			Intercept&0.3   &  0.2862  &  0.0054& 0.0138 &(0.2756, 0.2968)\\	
			$t/500$&0.0001 &   0.000102  &  0.00001  & $2\times 10^{-06}$ &(0.00008, 0.00012)\\
			cos(2$\pi t$/52)&0.2  &   0.2012  &  0.0029  &0.0012 & (0.1954, 0.2069)\\
			sin(2$\pi t$/52)&-0.4  &    -0.3954  &  0.0029 & 0.0046 &(-0.4012, -0.3896)\\
			MA1&-3  &  -3.0192  &  0.0067  & 0.0192&(-3.0323, -3.0061)\\
			MA3&-2  &   -2.0140  &  0.0051 &0.0140 & (-2.0240, -2.0040)\\
			Intercept&0.1  &   0.0922  &  0.0040  & 0.0078 &(0.0844, 0.0999)\\
			cos(2$\pi t$/52)&-0.4  &  -0.4054  &  0.0053  & 0.0054 &(-0.4158, -0.3950)\\
			sin(2$\pi t$/52)&-0.5  &  -0.5013 &   0.0054 & 0.0013 &(-0.5119, -0.4907)\\
			Overdisp&2 &   2.1695  &  0.0118  & 0.1695 &(2.1463, 2.1927)\\
			
			\hline
			
		\end{tabular}
	\end{center}
\end{table}

\begin{figure}[h!]
	\begin{subfigure}{.33\linewidth}
		\centering
		\includegraphics[width=1\linewidth]{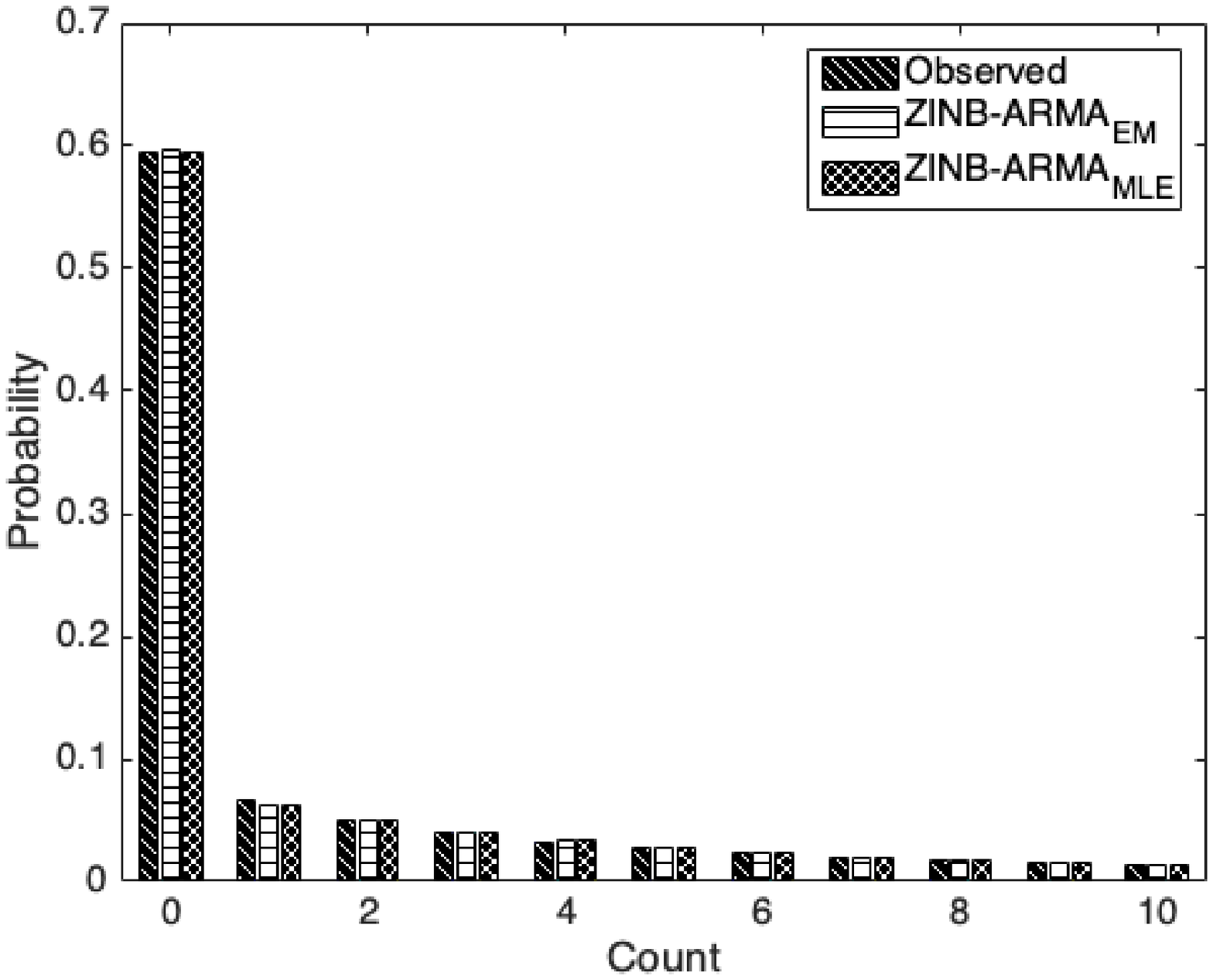}
		\caption{Model 1}
		\label{}
	\end{subfigure}%
	\begin{subfigure}{.33\linewidth}
		\centering
		\includegraphics[width=1\linewidth]{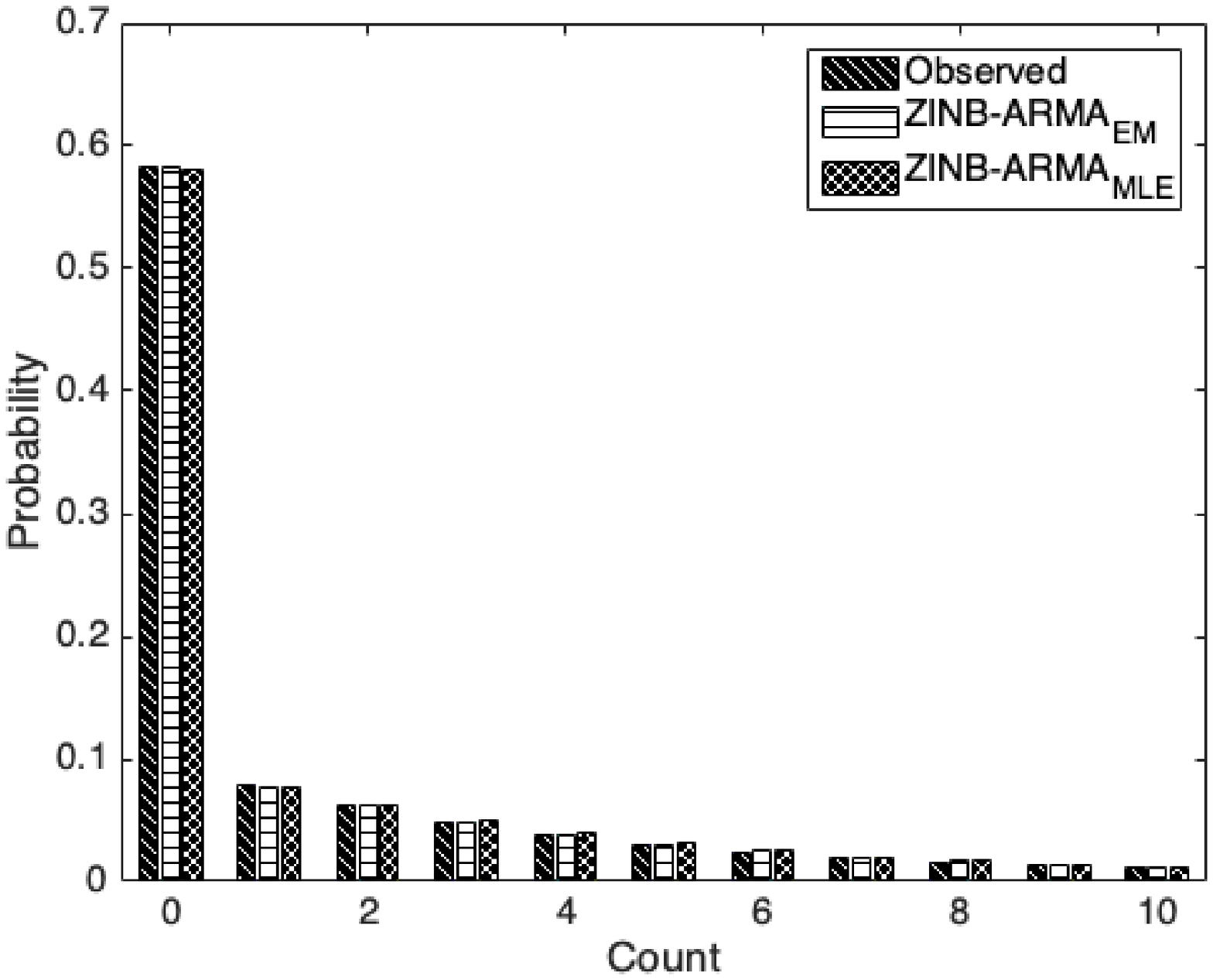}
		\caption{Model 2}
		\label{}
	\end{subfigure}
\begin{subfigure}{0.33\linewidth}
	\centering
	\includegraphics[width=1\linewidth]{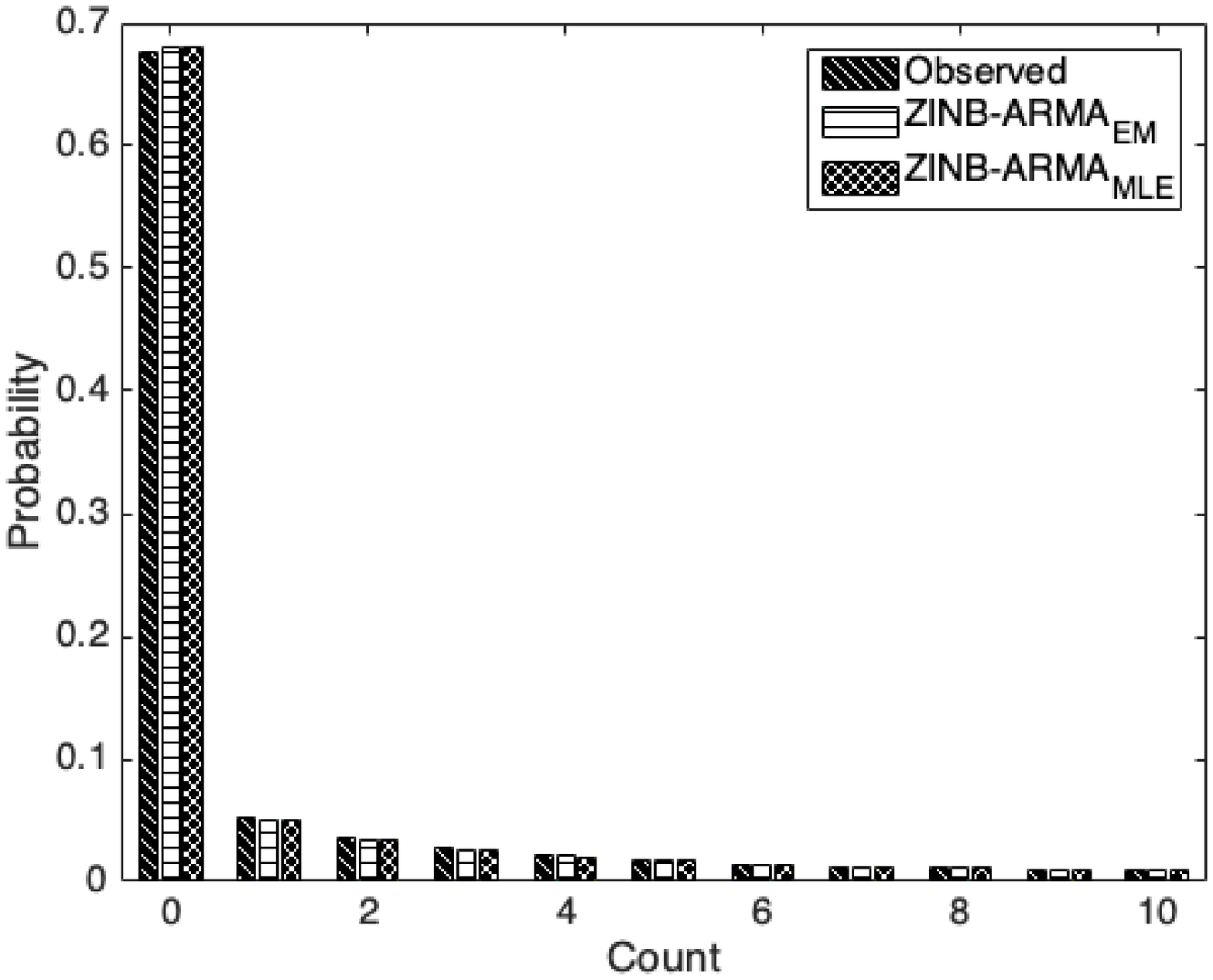}
	\caption{Model 3}
	\label{GF_sim2}
\end{subfigure}
	\caption{Goodness-of-fit plots for $n=100$ corresponding to simulation results  of Models 1, 2 and 3 considered in Section \ref{sim}. Here ZINB-ARMA denotes our proposed model of zero inflated NB-ARMA.}
	\label{GF_sim}
\end{figure}

\begin{figure}[h!]
	\centering
	\includegraphics[scale=0.40]{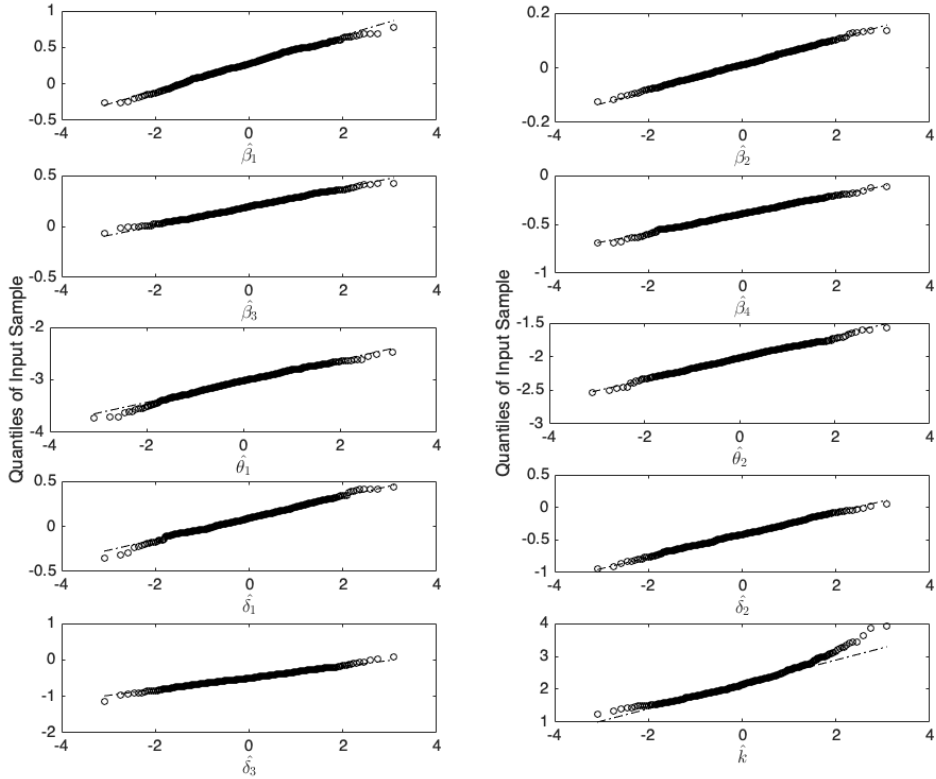}
	\caption{Normal probability plot i.e., QQ plot of the MLE estimates for simulated Model 3 in Section \ref{sim}, $N=100$.  Here $\beta_{1} = 0.3,\beta_{2} = 0.1,\beta_{3} = 0.2,\beta_{4} = -0.4$, $ \theta_1=-3,\theta_2=-2$, $\delta_{1}= 0.1,\delta_{2}= -0.4,\delta_{3}= -0.5,k= 2$.}
	\label{empqqplot}
\end{figure}

\begin{table}[h!]\small
	\centering
	\caption{Summary table for Models 1-6 in Section \ref{real}}
	\label{modelselectiontable}
	\begin{tabular}{|c|c|c|c|c|c|c|c|c|}
		\hline
		Model &$\#$ parameters  & MSE  & $\chi^{2}$ & df        & Deviance  & AIC      & BIC       \\ \hline
		M1&	16&	9.4866&	189.9052	&133&	108.2838&	481.7745&	529.8376 \\ 
		
		M2	&14	&10.4535&	199.0972&	135	&111.9753&	485.1197&	527.1749\\
		
		M3 &16	&10.5009&	204.5800	&133	&112.0686	&492.4039&	540.4671\\
		
		M4& 13	&11.6191&	200.7068&	136	&111.2048&	485.0038&	524.0551\\
		
		M5&10&	12.9951&	190.9382	&139&	109.8256&	480.7847&	510.8241\\
		M6&8	&12.7397	&189.2973&	141	&109.9701&	485.9513&	511.4829\\
		\hline
		
	\end{tabular}
\end{table}

\begin{table}[h!]
	\scriptsize
	\begin{center}
		\caption{EM Estimates of model M5 in Section \ref{real}}
		\label{tab:table5}
		\begin{tabular}{|c|c|c|c|c|r} 
			\hline
			Means &Parameters &Estimates & Std. Errors & p-values  \\
			\hline
			
			    &	Intercept &	2.0274   & 0.8010 &   0.0125\\
			&	$t'/148$&1.2642   & 0.2866   & $<0.0001$\\
			&	cos$ (2\pi t'/52) $&0.9644  &  0.1812 &   $<0.0001$\\
			$\lambda_t$	& sin$ (2\pi t'/52) $&-1.8482 &   0.1508  &      $< 0.0001$\\
			&	HMD	&-0.0348   & 0.0108   & 0.0016\\
			&	MA1&-0.0935   & 0.0759   & 0.2197\\
			&	MA3&0.1673   & 0.0575   & 0.0042\\
			&	Intercept &	-14.8254 &   8.6312  &  0.0881\\
			$\pi_t$&	cos$ (2\pi t'/52) $ &12.9723 &   8.2209  &  0.1168\\
			&	sin$ (2\pi t'/52) $ &6.8231  &  4.3688  &  0.1206\\\hline
			&	overdisp. &	8.0569   & 3.1593 &   0.0118\\
			\hline		\end{tabular}
	\end{center}
\end{table}

\begin{table}[h!]
	\label{t3}
	\scriptsize
	\begin{center}
		\caption{Summary table for all compared models in Section \ref{real}}
		\label{tab:denguedatacomparetable}
		\begin{tabular}{|cccccc|} 
			\hline
			Model& $\#$ Parameters &	 MSE&	AIC&	BIC  & MAD \\
			\hline
			M5 &	10&	12.99& 480.78&	510.82 &1.5315\\
			AM-1 &	9	&13.64&	487.76&	517.74& 1.7423\\
			AM-2 &	6 &	13.91&	494.88&	513.35 & 1.8156\\
			AM-3 & 4 & 16.19 & 863.14 & 899.19&2.4246\\
			AM-4 & 4  &14.53  & 870.34  & 906.39&2.2270\\
			\hline	
		\end{tabular}
	\end{center}
\end{table}
\begin{table}[h!]
	\scriptsize
	\begin{center}
		\caption{Sensitivity and Specificity Values for M5 and AM-1 considered in Section \ref{real}. ``Th'' represents threshold.}
		\label{ss}
		\begin{tabular}{|c|c|c|c|c|c|c|} 
			\hline
			Model&	\multicolumn{2}{|c|}{Th=0.4} & \multicolumn{2}{|c|}{Th=0.5} &\multicolumn{2}{|c|}{Th=0.6}  \\
			\hline
			& Sensitivity & Specificity &  Sensitivity & Specificity &  Sensitivity & Specificity \\
			
			\hline
			M5 & 92\%& 85\% &100\% &83\% &100\% &81\%\\
			AM-1 & 60\% & 85\% & 68\% &81\% & 74\% &78\% \\
			
			\hline
			
		\end{tabular}
	\end{center}
\end{table}

\begin{figure}[h!]
	\centering
	\includegraphics[scale=0.30]{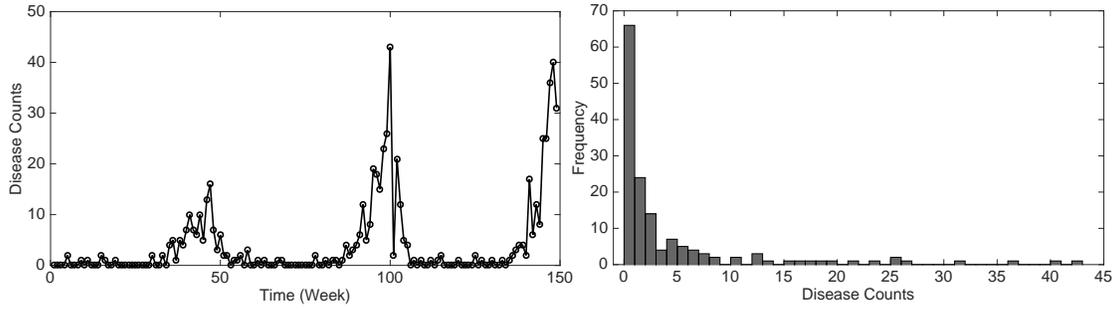}
	\caption{Time Series plot and histogram of weekly dengue count data}
	\label{data_hist}
\end{figure}
\begin{figure}[h!]
	\begin{subfigure}[b]{.5\textwidth}
		\centering
		\includegraphics[width=.8\linewidth]{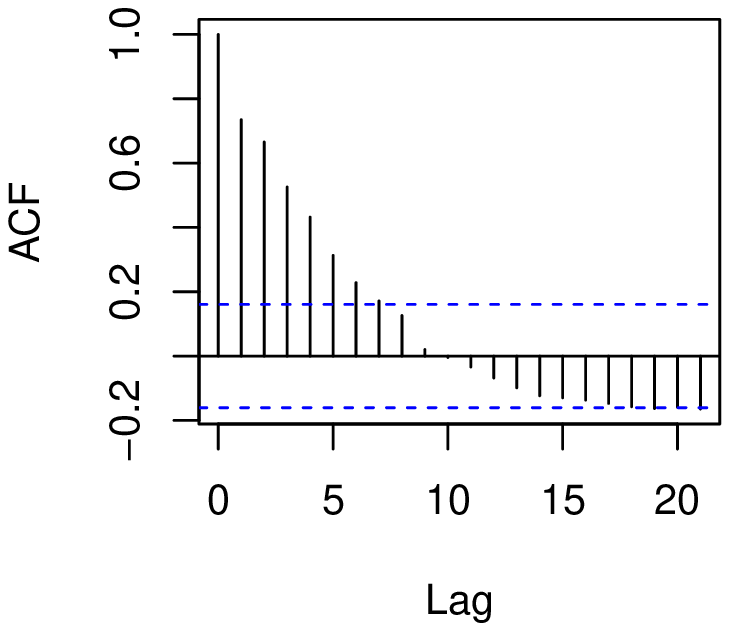}
		\caption{ACF plot}
		\label{}
	\end{subfigure}%
	\begin{subfigure}[b]{.5\textwidth}
		\centering
		\includegraphics[width=.8\linewidth]{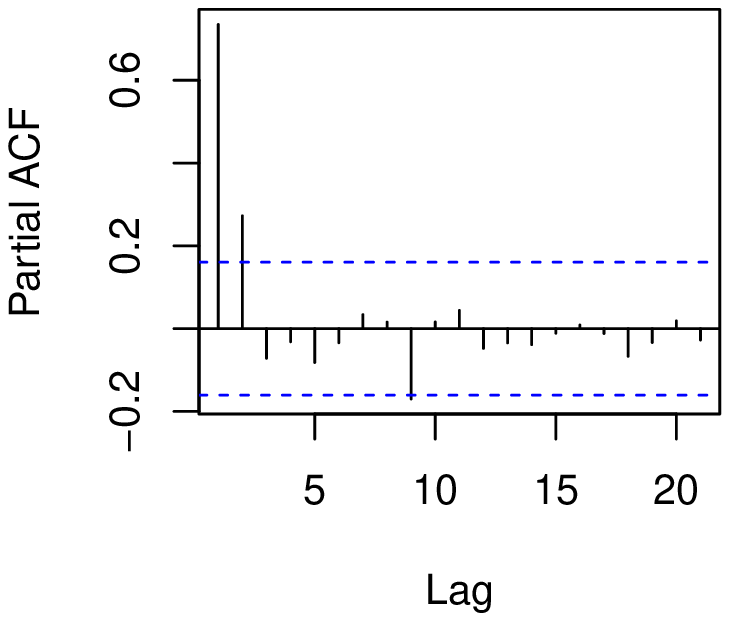}
		\caption{PACF plot}
		\label{}
	\end{subfigure}
	\caption{ACF and PACF plots of raw weekly dengue counts}
	\label{acfdata}
\end{figure}
\begin{figure}[h!]
	\begin{subfigure}{.5\linewidth}
		\centering
		\includegraphics[width=.7\linewidth]{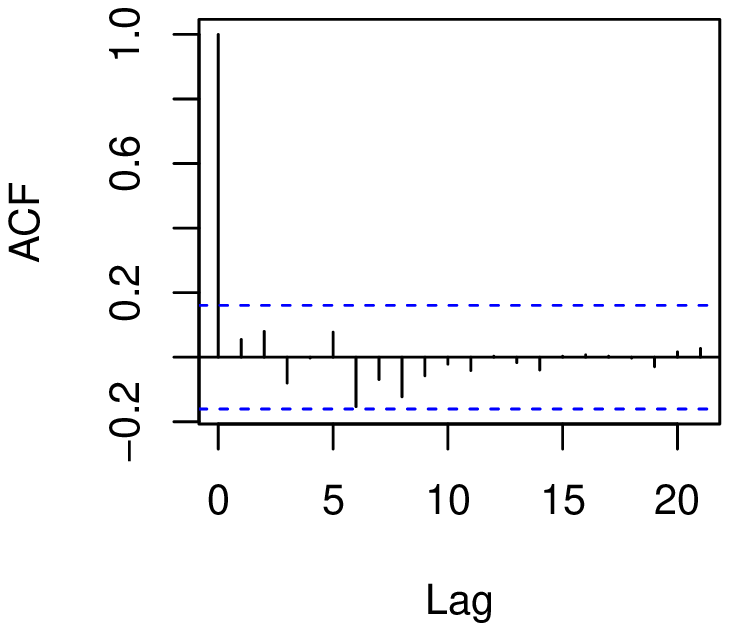}
		\caption{ACF plot of quantile residuals}
		\label{}
	\end{subfigure}%
	\begin{subfigure}{.5\linewidth}
		\centering
		\includegraphics[width=.7\linewidth]{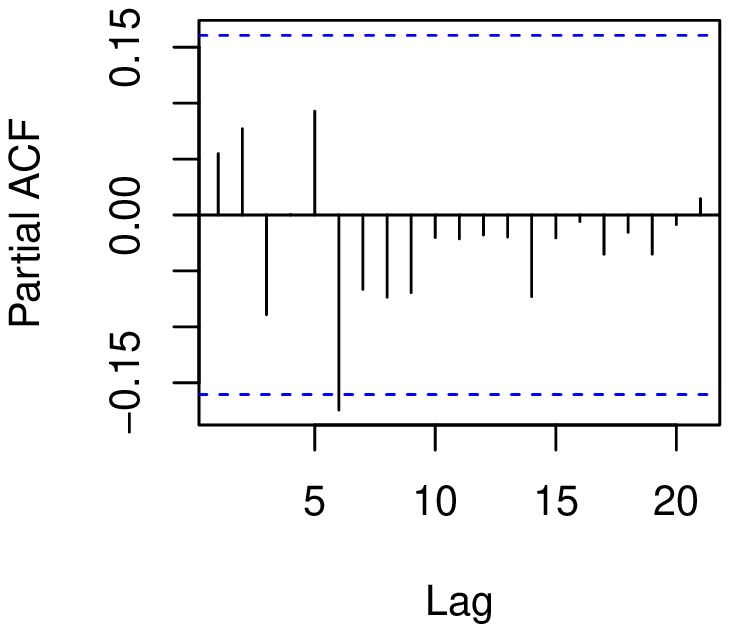}
		\caption{PACF plot of quantile residuals}
		\label{}
	\end{subfigure}
\begin{subfigure}{\linewidth}
	\centering
	\includegraphics[scale=0.3]{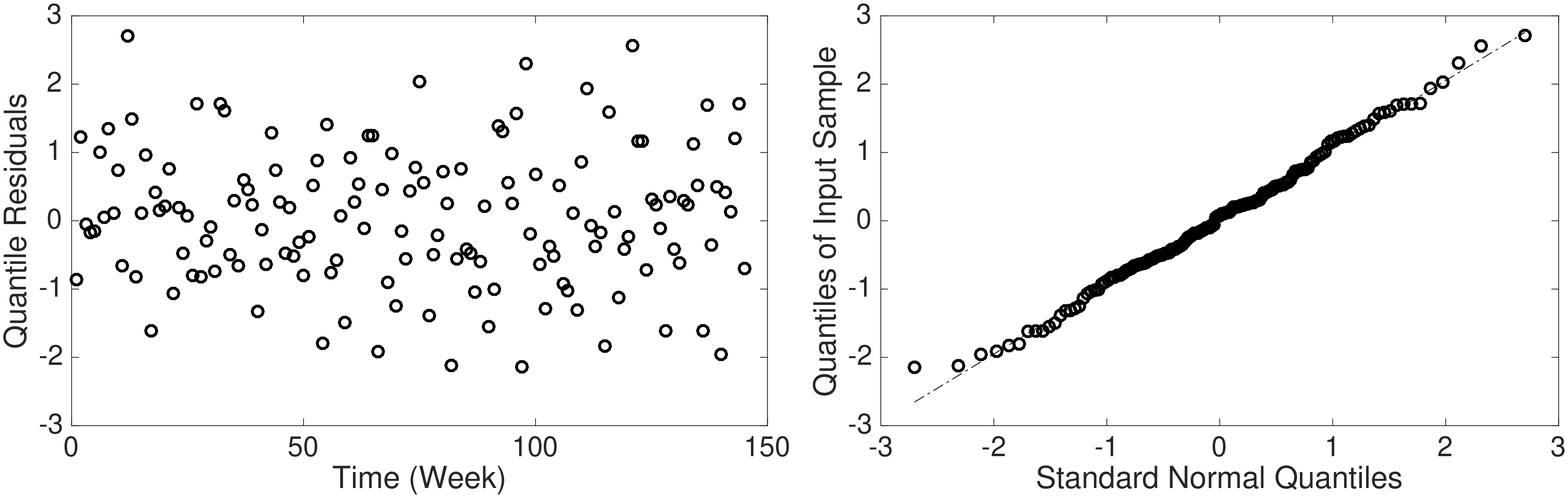}
	\caption{Randomized quantile residuals and QQ-plot}
	\label{residual_qqplot}
\end{subfigure}
	\caption{Randomized quantile residual analysis for model M5 in Section \ref{real}}
	\label{acfresiduals}
\end{figure}
\begin{figure}[h!]
	\centering
	\includegraphics[scale=0.30]{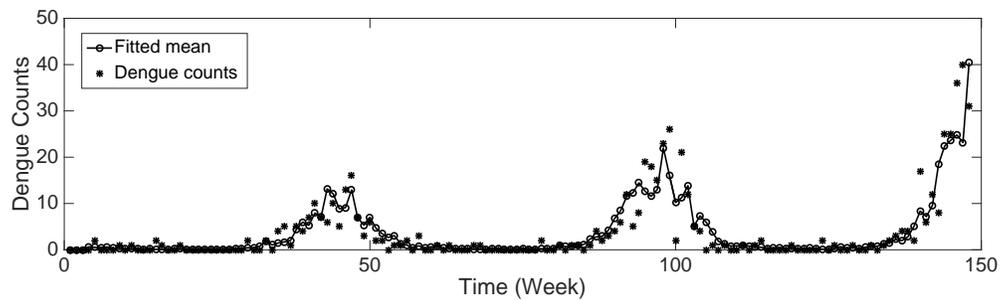}
	\caption{Actual dengue counts and fitted mean values from Model 5}
	\label{counts_fitted}
\end{figure}



\begin{table}[h!]\small
	\centering
	\caption{Summary table for Models 1-6 corresponding to dengue data in Section \ref{real2}}
	\label{modelselectiontable2}
	\begin{tabular}{|c|c|c|c|c|c|c|c|c|}
		\hline
		Model &$\#$ parameters  & MSE  & $\chi^{2}$ & df        & Deviance  & AIC      & BIC       \\ \hline
		M1&4&20.7445&	190.8015&	205	&143.3237	&1080.6572&	1097.3698\\
		M2&5&	20.1038	&190.5589&	204&	144.3150	&1075.0041	&1091.7157\\
		M3&6&	19.8279	&188.6323	&203	&144.7702&	1072.6851	&1092.5425\\
		M4&7&	19.9974	&188.6969&	202	&144.5563	&1079.1456	&1092.7394\\	
		M5&6&	20.2616	&191.2593	&203&	144.6008&	1077.9771&	1098.0313\\
		M6&7&	19.8713	&188.8358	&202&	144.9468	&1075.3745&	1098.7672\\
		\hline
	\end{tabular}
\end{table}

\begin{table}[h!]
	\scriptsize
	\begin{center}
		\caption{EM estimates of zero-inflated NB-ARMA fitted to Syphilis data, in Section \ref{real2}}
		\label{tab:estimates}
		\begin{tabular}{|c|c|c|c|c|r}
			\hline
			Means &Parameters &Estimates & Std. Errors & p-values  \\
			\hline
			
			&	Intercept &1.6134   & 0.0949   &      0\\
			$\lambda_t$	&	$t'/208$&0.3775   & 0.1564 &   0.0167\\
			& MA1&-0.1509  &  0.0608  &  0.0138\\
			&	MA2&-0.0724  &  0.0562  &  0.1995\\\hline
			$\pi_t$&	Intercept &	-1.3091  &  0.3724  &  0.0005\\
			&	$t'/208$&0.2122  &  0.6121  &  0.7291\\\hline
			&	overdisp. &3.0981  &  0.6864  &  0.0000\\
			\hline		\end{tabular}
	\end{center}
\end{table}


\begin{table}[h!]\small
	\centering
	\caption{Model comparisons for Syphilis data, in Section \ref{real2}}
	\label{Compare_ghahramani}
	\begin{tabular}{|c|c|c|c|c|c|}
		\hline
		Model & MSE  & AIC      & BIC       \\ \hline
		
		Zero inflated NB-ARMA&19.8279	&	1072.6851	&1092.5425\\
		ZIP& 20.2455 &1173.6862&	1187.0551\\
		ZINB&20.2326	&1077.3358	&1097.3898\\
		FRM Method&	20.2468	&	1078.3750	&1098.4290\\
		\hline
	\end{tabular}
\end{table}

%


\end{document}